\documentclass{amsart}
\usepackage[all]{xy}
\usepackage{color}
\usepackage{amssymb}
\let\cal\mathcal

\def\Bscr{{\cal B}}

\def\Dscr{{\cal D}}
\def\Escr{{\cal E}}
\def\Fscr{{\cal F}}
\def\Gscr{{\cal G}}

\def\Kscr{{\cal K}}
\def\Lscr{{\cal L}}
\def\Mscr{{\cal M}}
\def\Nscr{{\cal N}}
\def\Oscr{{\cal O}}

\def\Tscr{{\cal T}}

\let\blb\mathbb

\def \ZZ{{\blb Z}}

\def \HH{{\blb H}}

\def\Dis{\operatorname{Dis}}

\def\Id{\operatorname{id}}
\def\pr{\mathop{\text{pr}}\nolimits}

\def\Der{\operatorname{Der}}

\def\ctimes{\mathbin{\hat{\otimes}}}

\def\Lotimes{\overset{L}{\otimes}}

\def\Mod{\operatorname{Mod}}

\def\gr{\operatorname{gr}}

\def\gr{\operatorname {gr}}

\def\Ext{\operatorname {Ext}}
\def\Hom{\operatorname {Hom}}
\def\uHom{\operatorname {\mathcal{H}\mathit{om}}}

\def\RHom{\operatorname {RHom}}
\def\uRHom{\operatorname {R\mathcal{H}\mathit{om}}}

\def\Tot{\operatorname {Tot}}

\def\r{\rightarrow}

\let\dirlim\injlim

\newtheorem{lemma}{Lemma}[section]
\newtheorem{proposition}[lemma]{Proposition}
\newtheorem{theorem}[lemma]{Theorem}

\theoremstyle{definition}

{

}

\theoremstyle{remark}

\newtheorem{remark}[lemma]{Remark}

\newdimen\uboxsep \uboxsep=1ex
\def\uboxn#1{\vtop to 0pt{\hrule height 0pt depth 0pt\vskip\uboxsep
\hbox to 0pt{\hss #1\hss}\vss}}

\def\uboxs#1{\vbox to 0pt{\vss\hbox to 0pt{\hss #1\hss}
\vskip\uboxsep\hrule height 0pt depth 0pt}}

\numberwithin{equation}{section}
\def\HH{\operatorname{HH}}
\def\poly{\operatorname{poly}}
\def\uExt{\mathop{\Escr\mathit{xt}}\nolimits}
\title[Hochschild (co)homology for Lie algebroids]{Hochschild (co)homology for Lie
algebroids}
\author{Damien Calaque}
\address[Damien Calaque]{Institut Camille Jordan,
Universit\'e Claude Bernard Lyon 1,
43 boulevard du 11 novembre 1918,
F-69622 Villeurbanne Cedex
France}
\email{damien.calaque@math.univ-lyon1.fr}
\author{Carlo A. Rossi}
\address[Carlo A. Rossi]{Department of mathematics, ETH Zurich, 8092 Zurich, Switzerland}
\email{carlo.rossi@math.ethz.ch}
\author{Michel van den Bergh}
\address[Michel Van den Bergh]{Departement WNI, Hasselt University, Agoralaan, 3590
Diepenbeek, Belgium}
\email{michel.vandenbergh@uhasselt.be}
\thanks{The research of the first
  author is partly
  supported by the french national agency through the ANR project G\'eSAQ (project number
$JC08\underline{~}320699$). }
\thanks{The third author is a director of research at the FWO.}
\begin{document}
\begin{abstract}
  We define the Hochschild (co)homology of a ringed space relative to
  a locally free Lie algebroid.  Our definitions mimic those of Swan
  and Caldararu for an algebraic variety. We show that our
  (co)homology groups can be computed using suitable standard
  complexes.

Our formul\ae\ depend on certain natural structures on jetbundles over
Lie algebroids. In an appendix we explain this by showing that
such jetbundles are formal groupoids which serve as the formal
exponentiation of the Lie algebroid.
\end{abstract}
\maketitle
\section{Introduction}
This is a companion note to \cite{CRVdB}. Throughout $k$ is a base field
of characteristic zero. If $X$ is a smooth algebraic variety
over $k$ of dimension $d$ then Caldararu defines the Hochschild (co)homology
of $X$ as
\begin{equation}
\label{ref-1.1-0}
\begin{split}
\HH^n(X)&=\operatorname{Ext}^n_{\mathcal{O}_{X\times
X}}(\mathcal{O}_\Delta,\mathcal{O}_\Delta)\\
\HH_n(X)&=\operatorname{Ext}^{d-n}_{\mathcal{O}_{X\times X}}(\omega^{-1}_\Delta,
\Oscr_\Delta)
\end{split}
\end{equation}
where $\Delta\subset X\times X$ denotes the diagonal. The first of
these definitions is due to Swan~\cite{Swan}.

From these definitions
it is clear that  $\HH^\bullet(X)$ has a canonical algebra structure
(by the Yoneda product) and $\HH_\bullet(X)$ is a module over it (by the action
of $\HH^\bullet(X)$ on $\Oscr_\Delta$). As customary we refer below
to these algebra and module structures as ``cup'' and ``cap'' products.

Building on the work of a number of people (notably Kontsevich and
Shoikhet) we completed in \cite{CRVdB} the proof of a conjecture by
Caldararu which asserts that there is a certain Duflo type isomorphism
between the above Hochschild (co)homology groups and the cohomology
groups of poly-vector fields and differential forms which preserves
the natural algebra and module structures. We refer to
\cite{Caldararu2,DTT2,DTT1} for background information and additional
results.

One small issue was left open. Instead of using
\eqref{ref-1.1-0} directly we used explicit chain and cochain complexes
for the definition of Hochschild (co)homology.  As a result it is not
immediately obvious that our algebra and module structures are
precisely the same as Caldararu's. The fact that this is true
for the cup product was proved in \cite{Y1} by Yekutieli.

In \cite{CRVdB} we actually proved a version of Caldararu's conjecture
valid for locally free Lie algebroids. This yields in particular the
algebraic, analytic and $C^\infty$-setting as special cases. In this
paper we prove in the Lie algebroid setting an agreement property (see Theorem
\ref{ref-11.1-44})
between the Hochschild (co)homology defined by complexes and by
formul\ae\  similar to \eqref{ref-1.1-0} (see \eqref{ref-4.2-17}).

Our formul\ae\  depend on various
interesting structures on the sheaf of jet bundles of a Lie
algebroid. In Appendix \ref{ref-A-45} we clarify this by showing that
these structures make the sheaf of jet bundles into a formal groupoid
which serves as the formal exponentiation of the Lie algebroid (see
also \cite[Appendix]{Kapranov3} and \cite[\S3.4]{KP}).

\section{Acknowledgement}

The authors thank the referee for his thorough reading of the paper as well
as for pointing out Proposition \ref{new6}.

\section{Notation and conventions}

Unadorned tensor products are over $k$. We usually write $\otimes_X$
instead of $\otimes_{\Oscr_X}$ and we apply a similar convention for $\Hom$.
We often drop ``sheaf of''. For example we usually speak of an algebra
instead of a sheaf of algebras.
Lower indices denote homological grading. If we need to translate
between homological and cohomological grading we use the convention
$H_n(-)=H^{-n}(-)$.

Some objects below come with a natural topology which will be appropriately
specified. If an object is introduced without a specific topology then it
is assumed to have the discrete topology. This applies in particular to
structure sheaves.

\section{Preliminaries}
\label{ref-3-1}
\subsection{Sites}
\label{new7}
For the theory of sites we refer to \cite{SGA41}. We freely use sheaf theory over
(ringed) sites and in particular the fact
that the category of modules over a ringed site is a Grothendieck
category (see \cite[Prop.\ II.6.7]{SGA41}). By definition this is an
abelian category with a generator and exact filtered colimits. Such a
category automatically has enough injectives and arbitrary products \cite{Groth1}.

We will also use the fact that the category of complexes over a ringed
site has both K-flat resolutions \cite[Theorem 3.4]{Spaltenstein} and
K-injective resolutions \cite{TLS}. Hence we may freely use unbounded
$\Hom$'s and tensor products and the corresponding $\Hom$-tensor
identities.
\section{Lie algebroids, enveloping algebras, jet bundles and connections}
\subsection{Lie algebroids}
 Throughout
$(X,\Oscr_X)$ is a ringed site (or ringed space if the reader is not
interested in the utmost generality) and $\Lscr$ is a Lie algebroid on
$X$ locally free of rank $d$.
By definition $\Lscr$ is a sheaf of Lie algebras acting on $\Oscr_X$
which is also an $\Oscr_X$-module satisfying the following conditions
\begin{equation}
\label{ref-3.1-2}
\begin{aligned}
(f_1l)(f_2)&=f_1 l(f_2)\\
l(f_1f_2)&=l(f_1)f_2+f_1l(f_2)\\
[l_1,l_2](f)&=l_1(l_2(f))-l_2(l_1(f))\\
[l_1,fl_2]&=l_1(f)l_2+f[l_1,l_2]
\end{aligned}
\end{equation}
for sections $f,f_1,f_2$ of $\Oscr_X$ and sections $l,l_1,l_2$ of $\Lscr$.
\subsection{Universal enveloping algebras}
The universal enveloping algebra  (see \cite{Rine}) of $\Lscr$ is denoted by
$\mathrm{U}_X\Lscr$.   To define this object note that $\Oscr_X\oplus \Lscr$
also carries the structure of a sheaf of Lie algebras via $[(f_1,l_1),(f_2,l_2)]
=(l_1(f_2)-l_2(f_1),[l_1,l_2])$.
Then $\mathrm{U}_X\Lscr$ is the quotient of the universal enveloping
algebra of $\Oscr_X\oplus \Lscr$ subject to the additional relation
$f\cdot l=fl$, for $f$ in $\Oscr_X$ and $l$ in $\Oscr_X\oplus \Lscr$.

If $X$ is a smooth algebraic variety and $\Lscr=\Tscr_X$ then
$\mathrm{U}_X\Lscr$ equals $\Dscr_X$, the sheaf of differential
operators on $X$. In general the properties of $\mathrm{U}_X\Lscr$
mimic those of $\Dscr_X$. In particular giving $\Oscr_X$ degree zero
and $\Lscr$ degree one, $\mathrm{U}_X\Lscr$ becomes equipped with an
ascending filtration $F^\bullet$ such that
\begin{equation}
\label{ref-3.2-3}
\gr_F \mathrm{U}_X\Lscr=\mathrm S_X\Lscr
\end{equation}  The
action of $\Lscr$ on $\Oscr_X$ extends to an action of
$\mathrm{U}_X\Lscr$ on $\Oscr_X$ which makes $\Oscr_X$ into a left
$\mathrm{U}_X\Lscr$-module.

As $\mathrm{U}_X\Lscr$ contains $\Oscr_X$ it is equipped with a
natural left $\Oscr_X$-action. We view $\mathrm{U}_X\Lscr$
as a central(!) $\Oscr_X$-bimodule with the right $\Oscr_X$-action
defined to be equal to the left one. In this way
$\mathrm{U}_X\Lscr$ becomes a sheaf of cocommutative $\Oscr_X$-coalgebras. More
precisely there is a comultiplication $\Delta:\mathrm{U}_X\Lscr \r
\mathrm{U}_X\Lscr\otimes_X \mathrm{U}_X\Lscr$ and a counit
$\epsilon:\mathrm{U}_X\Lscr\r \Oscr_X$ which are locally given by the
following formul\ae\  (using the Sweedler convention)
\begin{equation}\label{ref-3.3-4}
\begin{aligned}
\Delta(f)&=f\otimes 1=1\otimes f\\
\Delta(l)&=l\otimes 1+1\otimes l\\
\Delta(DE)&=\sum_{D,E}D_{(1)}E_{(1)}\otimes D_{(2)}E_{(2)}\\
\epsilon(D)&=D(1)
\end{aligned}
\end{equation}
for $f$ a section of $\Oscr_X$, $l$ a section of $\Lscr$ and $D,E$
sections of $\mathrm{U}_X\Lscr$. Although $\mathrm{U}_X\Lscr\otimes_X
\mathrm{U}_X\Lscr$ is not a sheaf of algebras the third formula is
well defined as $\Delta$ takes values in a certain subsheaf
of $\mathrm{U}_X\Lscr\otimes_X
\mathrm{U}_X\Lscr$ which is an algebra (see e.g.\ \cite{Xu}).

%We further consider a smooth groupoid scheme $\mathcal G=\mathcal G(G_1,G_0)$, where
%$G_0=X$, $X$ as above; we then denote by $\mathcal L=\mathcal L_\mathcal G$ the
%corresponding Lie algebroid.

%In this case, $\mathrm J_X\mathcal L=s^*\left(\widehat{\mathcal O}_{G_1,G_0}\right)$,
%where $s$ denotes the source map of $\mathcal G$, and $G_0$ is regarded as a subscheme of $G_1$ {\em via} the unit map $e$ of $\mathcal G$.
%Finally, we notice that the pairing between $\mathrm J_X\mathcal L$ and $\mathrm U_X\mathcal L$ is simply recovered from the fact that,
%in this setting, $\mathrm U_X\mathcal L$ is regarded as the sheaf of translation-invariant differential operators on $G_1$.

\subsection{Jet bundles}
\label{ref-3.3-5}
The sheaf of $\Lscr$-jets on $X$ is defined as
\begin{equation}
\label{ref-3.4-6}
{\mathrm J}_X\Lscr=\uHom_{X}(\mathrm U_X\Lscr,\Oscr_X)
\end{equation}
(this is unambiguous, as the left and right $\Oscr_X$-modules structures on
$\mathrm U_X\Lscr$ are the same).  Being the dual of an
$\Oscr_X$-module ${\mathrm J}_X\Lscr$ is also an $\Oscr_X$-module (given that
$\Oscr_X$ is commutative).  Below we will sometimes use the
corresponding $\Oscr_X$-linear evaluation pairing
\begin{equation}
\label{ref-3.5-7}
\langle-,-\rangle: {\mathrm J}_X\Lscr\times \mathrm U_X\Lscr\r \Oscr_X
\end{equation}
The cocommutative coalgebra structure on
$\mathrm U_X\Lscr$ induces a commutative algebra structure on ${\mathrm J}_X\Lscr$ by the
usual formula
\begin{equation}
\label{ref-3.6-8}
(\alpha\beta)(D)=\sum_D \alpha(D_{(1)})\beta(E_{(2)})
\end{equation}
for $\alpha,\beta$ sections on ${\mathrm J}_X\Lscr$ and $D$ a section of
$\mathrm U_X\Lscr$. The unit ``1'' of ${\mathrm J}_X\Lscr$ is given by~$\epsilon$. One
verifies that $\Oscr_X\r {\mathrm J}_X\Lscr:f\mapsto f\cdot 1$ is an algebra
homomorphism. So ${\mathrm J}_X\Lscr$ is an $\Oscr_X$-algebra.

The natural ascending filtration $F^\bullet$ on $\mathrm U_X\Lscr$
introduced above induces a descending filtration $F_\bullet$ on
${\mathrm J}_X\Lscr$ where $F_n {\mathrm J}_X\Lscr$ is given by those sections of
${\mathrm J}_X\Lscr=\uHom_{X}(\mathrm U_X\Lscr,\Oscr_X)$ which vanish on $F^n \mathrm
U_X\Lscr$.

One checks by a local computation that $F_\bullet$ is the adic filtration
for the ideal $\mathrm J^c_X\Lscr=F_1{\mathrm J}_X\Lscr\subset {\mathrm J}_X\Lscr$. For
this adic filtration
${\mathrm J}_X\Lscr$ is complete and furthermore we have
\begin{equation}
\label{ref-3.7-9}
\gr {\mathrm J}_X\Lscr=\mathrm S_X\Lscr^\ast
\end{equation}
Locally we may lift a basis $x_1,\ldots,x_d$ for $\Lscr^\ast$ to $\mathrm J^c\Lscr$
and in this way one obtains a local isomorphism of sheaves of algebras
\begin{equation}
\label{new5}
{\mathrm J}_X\Lscr\cong\Oscr_X[[x_1,\ldots,x_d]]
\end{equation}
\begin{lemma}
\label{ref-3.1-10}
If we equip $\mathrm U_X\Lscr$ with the discrete topology
and ${\mathrm J}_X\Lscr$ with the ${\mathrm J}_X^c\Lscr$-adic topology then
\eqref{ref-3.5-7}
is a non-degenerate pairing of sheaves of topological $\Oscr_X$-modules in the sense that
it induces isomorphisms
\begin{align}
{\mathrm J}_X\Lscr&=\uHom_X(\mathrm U_X\Lscr,\Oscr_X)\\
\mathrm U_X\Lscr&=\uHom_X^{\text{\textrm{cont}}}({\mathrm
J}_X\Lscr,\Oscr_X)\label{ref-3.9-11}
\end{align}
\end{lemma}
\begin{proof}
The first isomorphism is by definition so we concentrate on the second one.

Note that $\uHom_{X}^{\text{cont}}({\mathrm J}_X\Lscr,\Oscr_X)\subset \uHom_{X}({\mathrm
J}_X\Lscr,\Oscr_X)$
is given by those sections which vanish (locally) on some power of $\mathrm J^c_X\Lscr$.
The pairing \eqref{ref-3.5-7} induces a pairing of locally
free $\Oscr_X$-modules of finite rank
\[
\langle-,-\rangle: {\mathrm J}_X\Lscr/({\mathrm J}_X^c\Lscr)^n\times F^n \mathrm
U_X\Lscr\r \Oscr_X
\]
and from \eqref{ref-3.2-3} and \eqref{ref-3.7-9} it follows easily
that this pairing is non-degenerate.

Thus
\[
F^n\mathrm U_X\Lscr=\uHom_{X}({\mathrm J}_X\Lscr/(\mathrm J^c_X\Lscr)^n,\Oscr_X)
\]
Taking the direct limit yields \eqref{ref-3.9-11}
\end{proof}
As a slight generalization we consider the %objects
%\begin{align*}
%(\mathrm U_X\Lscr)^{\otimes_X n}&=\underbrace{\mathrm U_X\Lscr\otimes_X\cdots\otimes_X\mathrm U_X\Lscr}_{\text{$n$ times}}\\
%({\mathrm J}_X\Lscr)^{\ctimes_X n}&=\underbrace{{\mathrm J}_X\Lscr\ctimes_X\cdots\ctimes_X {\mathrm J}_X\Lscr}_{\text{$n$ times}}
%\end{align*}
%as well as
pairing
\begin{multline*}
\langle-,-\rangle: ({\mathrm J}_X\Lscr)^{\ctimes_X n}\times(\mathrm U_X\Lscr)^{\otimes_X
n} \r \Oscr_X:
\\(\alpha_1\otimes\cdots\otimes \alpha_n,D_1\otimes\cdots\otimes D_n)\mapsto
\langle \alpha_1,D_1\rangle\cdots
\langle \alpha_n,D_n\rangle
\end{multline*}
The filtrations $F^\bullet$ and $F_\bullet$ on $\mathrm U_X\Lscr$ and
${\mathrm J}_X\Lscr$ induce corresponding filtrations on
$(\mathrm U_X\Lscr)^{\otimes_X n}$ and
$({\mathrm J}_X\Lscr)^{\ctimes_X n}$ and the filtration on $({\mathrm
J}_X\Lscr)^{\ctimes_X n}$ is complete.
As in Lemma \ref{ref-3.1-10}  one shows that $\langle-,-\rangle$
is non-degenerate.
\subsection{Flat connections}
\label{ref-3.4-12}
If $\Mscr$ is an $\Oscr_X$-module then an $\Lscr$-connection on $\Mscr$ is a map
\[
\nabla:\Lscr\otimes_k \Mscr\r \Mscr
\]
with properties mimicking those of ordinary connections (which correspond
to $\Lscr=\Tscr_X$). Namely
\begin{align*}
\nabla_{fl}(m)&=f\nabla_l(m)\\
\nabla_{l}(fm)&=l(f)m+f\nabla_l(m)
\end{align*}
for sections $f$ of $\Oscr_X$, $l$ of $\Lscr$ and $m$ of $\Mscr$\footnote{Equivalently,
an $\Lscr$-connection on $\Mscr$ is determined by a map
${\rm d}^\nabla:\Mscr\to\Lscr^*\otimes_X\Mscr$ satisfying a Leibniz type identity (see
e.g.~\cite{vdbcalaque}).}. Here and below we make
use of the standard notation
$\nabla_l(m)=\nabla(l\otimes m)$. A
connection is flat if
$\nabla_{[l_1,l_2]}=\nabla_{l_1}\nabla_{l_2}-\nabla_{l_2}\nabla_{l_1}$. All
connections below are flat. A flat connection on $\Mscr$ extends to a left
$\mathrm U_X\Lscr$-module structure on $\Mscr$, and in fact this construction is
reversible yielding an equivalence between the two notions. If $D$ is
a section of $\mathrm U_X\Lscr$ then we sometimes denote its action
on a module with a flat connection by $\nabla_D$.

Clearly $\Oscr_X$ and $\mathrm U_X\Lscr$ are equipped with canonical
flat connections
\begin{align*}
{}^G\nabla_l f&=l(f)\\
{}^G\nabla_l D&=lD
\end{align*}
for sections $f$ of $\Oscr_X$, $l$ of $\Lscr$ and $D$ of $\mathrm U_X\Lscr$.

If $M,N$ are equipped with a flat $\Lscr$-connection then the same holds
for $M\otimes_X N$ and $\uHom_X(M,N)$. The formul\ae\  are the same as in
the case $\Lscr=\Tscr_X$.  This applies in particular to the
definition of ${\mathrm J}_X\Lscr$ \eqref{ref-3.4-6}.  Thus ${\mathrm J}_X\Lscr$ is also
equipped with a canonical flat connection which we denote by ${}^G\nabla$
as well.\footnote{The ``G'' stands for Grothendieck, as this connection
is often referred to as the ``Grothendieck connection.''} Explicitly for a section $l$ of
$\Lscr$, a section $\alpha$
of ${\mathrm J}_X\Lscr$ and a section $D$ of $\mathrm U_X\Lscr$ we have
\begin{align*}
{}^G\nabla_l(\alpha)(D)&=l(\alpha(D))-\alpha(lD)
\end{align*}
One verifies in particular
\begin{equation}
\label{ref-3.10-13}
{}^G\nabla_l(\alpha\beta)={}^G\nabla_l(\alpha)\beta+\alpha {}^G\nabla_l(\beta)
\end{equation}
Besides the left $\mathrm U_X\Lscr$-module on ${\mathrm J}_X\Lscr$ induced by ${}^G
\nabla$
there is another left $\mathrm U_X\Lscr$-action on ${\mathrm J}_X\Lscr$ which we
denote by ${}^2\nabla$. For sections $D,E$ of $\mathrm U_X\Lscr$ and $\alpha$ of
${\mathrm J}_X\Lscr$ we put
\[
({}^2\nabla_E \alpha)(D)=\alpha(DE)
\]
It is an easy verification that ${}^G\nabla$ and ${}^2\nabla$ commute.
See Appendix \ref{ref-A-45} for more details.

\medskip

If $X$ is a smooth algebraic variety and $\Lscr=\Tscr_X$ then we can
make the above definitions more concrete. As already mentioned above $\mathrm
U_X\Lscr$ is the sheaf of differential operators $\Dscr_X$ on $X$. We
also have ${\mathrm J}_X\Lscr=\pr_{1\ast}\widehat{\Oscr}_{X\times X,\Delta}$ and
\begin{equation}
\label{ref-3.11-14}
\begin{aligned}
\langle f\boxtimes g,D\rangle&=f D(g)\\
{}^G\nabla_D(f\boxtimes g)&=D(f)\boxtimes g\\
{}^2\nabla_D(f\boxtimes g)&=f\boxtimes D(g)
\end{aligned}
\end{equation}
for sections $f,g$ of $\Oscr_X$ and $D$ of $\Dscr_X$. The first line
refers to the pairing between ${\mathrm J}_X\Lscr$ and $\mathrm U_X\Lscr$ as
in \eqref{ref-3.5-7}.
\begin{remark}
This example is a special case of the following one: consider a smooth
groupoid scheme $\mathcal G=\mathcal G(G,X,s,t,e,\mu)$ over $X$ where
$s,t:G\rightarrow X$ are respectively the source and target maps,
$e:X\rightarrow G$ is the unit map and $\mu:G\times_{s,X,t} G\rightarrow G$
is the composition.

If $x\in X$, $g\in t^{-1}x$ and $u$ is a section of $\mathcal
O_{t^{-1}x}$ then we put $(L_gu)(h)=u(gh)$. This definition is such
that $(L_gu)(h)$ is defined when $t(h)=s(g)$. In other words
$L_gu$ is a function on $t^{-1}s(g)$. Thus $L_g$ maps sections of
$\mathcal O_{t^{-1}t(g)}$ to sections of $\mathcal O_{t^{-1}s(g)}$.

Let us write $\mathcal T_t\subset \mathcal T_{G}$ for the relative
tangent bundle of $t:G\rightarrow X$. The vector fields in $\mathcal
T_t$ act by derivations on $\mathcal O_{t^{-1}x}$ for any $x\in X$.
We say that a vector field $\xi$ in $\mathcal T_t$ is \emph{left
  invariant} if for any $g\in G$ and for any section $u$ of $\mathcal
O_{t^{-1}t(g)}$ we have $\xi(L_gu)=L_g\xi(u)$. It is easy to see that
the left invariant sections of $s_\ast\mathcal T_t$ are closed under
Lie brackets of vector fields and hence they form a Lie algebroid
on~$X$. By definition this is the Lie algebroid associated to
$\mathcal G$ and it is denoted by~$\mathcal L_{\mathcal G}$.

%One verifies that its enveloping algebra $\mathrm U_{\mathcal
%  G} \mathcal L_{\mathcal G}$ consists of the left invariant
%differential operators inside $s_\ast\mathcal D_t$. Left invariant differential
%operators are defined in the same as left invariant vector fields.

In this setting $\mathrm J_X\mathcal L=s_\ast\widehat{\mathcal
    O}_{G,X}$ where $X$ is regarded as a subscheme of $G$ via
the unit map $e$.  Vector fields on $G$ act on
$s_\ast\widehat{\mathcal O}_{G,X}$ by derivations. The Grothendieck
connection ${}^G\nabla$ is the restriction of
this action to the left invariant vector fields.

%Finally, we notice that
%the pairing between $\mathrm J_X\mathcal L$ and $\mathrm U_X\mathcal
%L$ is simply recovered from the fact that, in this setting, $\mathrm
%U_X\mathcal L$ is regarded as the sheaf of translation-invariant
%differential operators on $G_1$.

If we put $G=X\times X$, $s(x,y)=x$, $t(x,y)=y$, $e(x)=(x,x)$ and $
\mu((w,y),(x,w))=(x,y) $ then the data $(G,X,s,t,e,\mu)$ form a
groupoid on $X$. One verifies that the left invariant vector fields are
precisely those vector fields which are obtained by pullback from the
first projection $X\times X\rightarrow X$. This gives an expression
for the Grothendieck connection which agrees with \eqref{ref-3.11-14}.
\end{remark}
\section{Hochschild (co)homology for Lie algebroids}
\label{ref-4-15}
 We need a fragment of the groupoid structure on ${\mathrm J}_X\Lscr$ (see
Appendix \ref{ref-A-45})
namely the counit
\begin{equation}
\label{ref-4.1-16}
\epsilon:{\mathrm J}_X\Lscr\r \Oscr_X:\alpha\mapsto\alpha(1)
\end{equation}
where the $1$ is the unit of $\mathrm U_X\Lscr$. The kernel
of $\epsilon$ is the sheaf of ideals $\mathrm J^c_X\Lscr$ introduced
above.

We use $\epsilon$ to make any $\Oscr_X$-module into a ${\mathrm J}_X\Lscr$-module. We
define the Hochschild (co)homology for $(X,\Oscr_X,\Lscr)$ as
\begin{equation}
\label{ref-4.2-17}
\begin{split}
\HH^n_{\Lscr}(X)&=\operatorname{Ext}^n_{{\mathrm J}_X\Lscr}(\mathcal{O}_X,\mathcal{O}_X)\\
\HH_n^{\Lscr}(X)&=\operatorname{Ext}^{d-n}_{{\mathrm
J}_X\Lscr}(\wedge^d\Lscr,\mathcal{O}_X)
\end{split}
\end{equation}
This definition is motivated by the following proposition
\begin{proposition} Assume that $X$ is a smooth
  algebraic variety of dimension $d$ and $\Lscr=\Tscr_X$. Then we have
  an isomorphism
\[
(\HH^n_{\Lscr}(X),\HH_n^{\Lscr}(X))\cong (\HH^n(X),\HH_n(X))
\]
compatible with the obvious algebra and module structures.
\end{proposition}
\begin{proof}
  From \eqref{ref-3.11-14}\eqref{ref-4.1-16} we obtain that $\epsilon$ is given by
  $\epsilon(f\boxtimes g)=fg$. Thus we get (taking into account
$\wedge^d\Tscr_X=\omega_X^{-1}$).
\begin{align*}
\HH^n_{\Lscr}(X)&=\Ext^n_{\widehat{\Oscr}_{X\times X,\Delta}}(\Oscr_\Delta,\Oscr_\Delta)\\
\HH_n^{\Lscr}(X)&=\Ext^n_{\widehat{\Oscr}_{X\times
X,\Delta}}(\omega^{-1}_\Delta,\Oscr_\Delta)
\end{align*}
The inclusion map $\Oscr_{X\times X}\r \widehat{\Oscr}_{X\times X,\Delta}$
induces maps
\begin{align*}
p:\Ext^n_{\widehat{\Oscr}_{X\times X,\Delta}}(\Oscr_\Delta,\Oscr_\Delta)
&\r \Ext^n_{\Oscr_{X\times X}}(\Oscr_\Delta,\Oscr_\Delta)\\
q:\Ext^n_{\widehat{\Oscr}_{X\times X,\Delta}}(\omega^{-1}_\Delta,\Oscr_\Delta)
&\r \Ext^n_{\Oscr_{X\times X}}(\omega^{-1}_\Delta,\Oscr_\Delta)
\end{align*}
Which are obviously compatible with algebra and module structures.
We will prove that $p,q$ are isomorphisms. The flatness
of $\widehat{\Oscr}_{X\times X,\Delta}$ over $\Oscr_{X\times X}$
implies that there are isomorphisms
\begin{align*}
\widehat{\Oscr}_{X\times X,\Delta} \Lotimes_{\Oscr_{X\times X}}\Oscr_{\Delta}&\cong
\Oscr_{\Delta}\\
\widehat{\Oscr}_{X\times X,\Delta} \Lotimes_{\Oscr_{X\times X}}\omega^{-1}_{\Delta}&\cong
\omega^{-1}_{\Delta}
\end{align*}
in $D(\Mod(\widehat{\Oscr}_{X\times X,\Delta}))$.
Hence we obtain using change of rings
\begin{align*}
\Ext_{\Oscr_{X\times X}}^n(\Oscr_\Delta,\Oscr_\Delta)
&=\Ext_{\widehat{\Oscr}_{X\times X,\Delta}}^n(
\widehat{\Oscr}_{X\times X,\Delta}\Lotimes_{\Oscr_{X\times X}}\Oscr_\Delta,\Oscr_\Delta)\\
&\cong \Ext_{\widehat{\Oscr}_{X\times X,\Delta}}^n(
\Oscr_\Delta,\Oscr_\Delta)
\end{align*}
and one easily checks that this isomorphism is the inverse of $p$. The morphism
$q$ is treated similarly.
\end{proof}
For the sequel the above definition of Hochschild homology is not so
convenient. We will modify it.
\begin{lemma} There is a canonical isomorphism in $D(\Mod({\mathrm J}_X\Lscr))$.
\begin{equation}
\label{ref-4.3-18}
\uRHom_{{\mathrm J}_X\Lscr}(\Oscr_X,{\mathrm J}_X\Lscr)=\wedge^d\Lscr[-d]
\end{equation}
\end{lemma}
\begin{proof}  We need to show
\[
\uExt^i_{{\mathrm J}_X\Lscr}(\Oscr_X,{\mathrm J}_X\Lscr)
=
\begin{cases}
\wedge^d\Lscr &\text{if $i=d$}\\
0&\text{otherwise}
\end{cases}
\]
First we establish this locally in the case that
${\mathrm J}_X\Lscr=\Oscr_X[[x_1,\ldots,x_d]]$. Let~$K_\bullet$ be the Koszul
resolution of $\Oscr_X$ as ${\mathrm J}_X\Lscr$ module with respect to the regular
sequence $(x_1,\ldots,x_d)$. Thus
$K_\bullet=\Oscr_X[[x_1,\ldots,x_d]][\xi_1,\ldots,\xi_d]$ where $(\xi_i)$ are variables
of degree $-1$ such that
$d\xi_i=x_i$. One computes
\[
\uExt^i_{{\mathrm J}_X\Lscr}(\Oscr_X,{\mathrm J}_X\Lscr)
=
\begin{cases}
\Oscr_X\xi_1^\ast\cdots \xi_d^\ast=\wedge^d \Lscr &\text{if $i=d$}\\
0&\text{otherwise}
\end{cases}
\]
One verifies that the resulting isomorphism
$ \uExt^d_{{\mathrm J}_X\Lscr}(\Oscr_X,{\mathrm J}_X\Lscr)\cong \wedge^d \Lscr$ is
independent of the choice of $(x_1,\ldots,x_d)$ and hence it globalizes.
\end{proof}
\begin{proposition}
We have a canonical isomorphism
\begin{equation}
\label{ref-4.4-19}
\HH_n^\Lscr(X)=\Ext^{d-n}_{{\mathrm J}_X\Lscr}(\wedge^d\Lscr,\Oscr_X)\cong
R^{-n}\Gamma(X,\Oscr_X\Lotimes_{{\mathrm J}_X\Lscr}\Oscr_X)
\end{equation}
compatible with the $\HH^\bullet_\Lscr(X)$ actions on the rightmost
copies of $\Oscr_X$.
\end{proposition}
\begin{proof}
We compute
\begin{align*}
\Ext^{d-n}_{{\mathrm J}_X\Lscr}(\wedge^d\Lscr,\Oscr_X)&=
R^{d-n}\Gamma(X,\uRHom_{{\mathrm J}_X\Lscr}(\uRHom_{{\mathrm J}_X\Lscr}(\Oscr_X,{\mathrm
J}_X\Lscr)[d],\Oscr_X))\\
&=R^{d-n}\Gamma(X,\Oscr_X\Lotimes_{{\mathrm J}_X\Lscr}\Oscr_X[-d])\\
&=R^{-n}\Gamma(X,\Oscr_X\Lotimes_{{\mathrm J}_X\Lscr}\Oscr_X)\qed
\end{align*}

As we have not touched the rightmost copy of $\Oscr_X$ on both sides
of \eqref{ref-4.4-19} it follows that this isomorphism is compatible
with the $\HH_\Lscr^\bullet(X)$-action.
\def\qed{}\end{proof}
\section{The Hochschild cochain complex}
The Hochschild cochain complex of $\Lscr$ (also called the sheaf of
$\Lscr$-poly-differential operators) $\mathrm{HC}^{\bullet}_{\Lscr,X}$ is defined as the
tensor algebra\footnote{In
  \cite{CRVdB} we used a shifted version of this complex (denoted by
$D_{\poly,X}^\Lscr$) to make the Lie
  bracket degree zero. Since here we emphasize the cup product we drop
  the shift.}  $T_X(\mathrm U_X\Lscr)$ with
differential
\[
\mathrm d_\mathrm H(D)=\begin{cases}
0,& p=0 \\
D\otimes 1-\Delta_p(D)+\Delta_{p-1}(D)-\cdots+(-1)^{p+1} 1\otimes D & p>0
%1\otimes_R D+\sum_{i=1}^p (-1)^i \Delta_i(D)+(-1)^{p+1} D\otimes_R 1 & p> 0.
\end{cases}
\]
where
$
D=D_1\otimes\cdots\otimes D_p
$
is a section of $T^p_X(\mathrm U_X\Lscr)$ and $\Delta_i$ is $\Delta$
applied to the $i$-th factor.  The Hochschild cochain complex
is naturally a DG-algebra with the product being derived from the standard
product in the tensor algebra $T_X(\mathrm U_X\Lscr)$. We refer
to this product as the ``cup product'' and denote it by $\cup$.
Explicitly we have
\[
(D_1\otimes\cdots\otimes D_p)\cup (E_1\otimes \cdots\otimes E_q)
=(-1)^{pq} D_1\otimes\cdots\otimes D_p\otimes E_1\otimes \cdots\otimes E_q
\]
\section{The Hochschild chain complex}
  The complex of $\Lscr$-poly-jets over $X$ is defined as
\[
\widehat{\mathrm{HC}}_{X,\bullet}({\mathrm J}_X\Lscr)=\bigoplus_{p\geq 0} ({\mathrm
J}_X\Lscr)^{\widehat{\otimes}_{X} p+1}
\]
equipped with the usual Hochschild differential
\[
\mathrm b_H(\alpha_0\otimes\alpha_1\otimes\cdots \otimes \alpha_p)
=\alpha_0\alpha_1\otimes\cdots\otimes\alpha_p-\alpha_0\otimes
\alpha_1\alpha_2\otimes\cdots\otimes\alpha_p+\cdots+
(-1)^p\alpha_p\alpha_0\otimes\cdots\otimes \alpha_{p-1}
\]
In other words, as implied by the notation,
$\widehat{\mathrm{HC}}_{X,\bullet}({\mathrm J}_X\Lscr)$ is simply the
(completed) relative Hochschild chain complex of the $\Oscr_X$-algebra ${\mathrm
J}_X\Lscr$.

By the usual Leibniz rule ${}^G\nabla$ acts on
$\widehat{\mathrm{HC}}_{X,\bullet}({\mathrm J}_X\!L)$ and one easily
verifies that the action of ${}^G\nabla$ commutes with $\mathrm b_H$.
In \cite{CRVdB} (following \cite{CDH}) we defined the Hochschild
chain complex\footnote{In \cite{CRVdB} we used the notation
  $\mathrm{HC}^\Lscr_{\poly,X}$ for $\mathrm{HC}^{\Lscr}_{X,\bullet}$}
$\mathrm{HC}^{\Lscr}_{X,\bullet}$ of $(X,\Oscr_X,\Lscr)$ as the
invariants of $\widehat{\mathrm{HC}}_{X,\bullet}({\mathrm J}_X\!L)$
under ${}^G\nabla$.  Explicitly for an object $U\r X$ of the site
\begin{align*}
\mathrm{HC}^{\Lscr}_{X,p}(U)&=\widehat{\mathrm{HC}}_{X,p}({\mathrm J}_X
\Lscr)(U)^{{}^G\nabla}\\
&=\{\alpha\in \widehat{\mathrm{HC}}_{X,p}({\mathrm J}_X \Lscr)(U)\mid \forall l\in
\Lscr(U):
{}^G\nabla_l(\alpha)=0\}
\end{align*}
%\begin{remark}
The reason for this somewhat roundabout way of defining the Hochschild
chain complex is technical. The idea is that the complicated formul\ae\ of
\cite{CR2}, valid for the ordinary Hochschild chain complex of
an algebra,  can be applied verbatim to
$\widehat{\mathrm{HC}}_{X,\bullet}({\mathrm J}_X\Lscr)$ which is also just an ordinary
(relative) Hochschild chain complex. We may then use the fact that these
formul\ae\  are invariant under ${}^G\nabla$ to descend them to
$\mathrm{HC}^{\Lscr}_{X,\bullet}$. This is a major work saving compared
to working directly with $\mathrm{HC}^{\Lscr}_{X,\bullet}$.
%\end{remark}

For use in the sequel we give a more direct description of
$\mathrm{HC}_{X,\bullet}^\Lscr$.
\begin{proposition}\label{ref-6.1-20}
We have as complexes
\begin{equation}
\label{ref-6.1-21}
\mathrm{HC}_{X,\bullet}^\Lscr\cong \bigoplus_{p\ge 0} ({\mathrm
J}_X\Lscr)^{\widehat\otimes_X p}
\end{equation}
with the differential on the right-hand side being given by
\begin{multline*}
\mathrm{b}_H(\alpha_1\otimes\cdots\otimes \alpha_p)=
\epsilon(\alpha_1)\alpha_2\otimes\cdots\otimes \alpha_p
-\alpha_1\alpha_2\otimes\cdots\alpha_p+\cdots \\\cdots+(-1)^{p-1}
\alpha_1\otimes\cdots\otimes\alpha_{p-1}\alpha_p+(-1)^{p}
\alpha_1\otimes\cdots\otimes\alpha_{p-1}\epsilon(\alpha_p)
\end{multline*}
The isomorphism \eqref{ref-6.1-21} is the restriction to
$\widehat{\mathrm{HC}}_{X,\bullet}({\mathrm J}_X
\Lscr)^{{}^G\nabla}=\mathrm{HC}^{\Lscr}_{X,\bullet}$ of the map
\begin{equation}
\label{ref-6.2-22}
\widehat{\mathrm{HC}}_{X,\bullet}({\mathrm J}_XL)\rightarrow \bigoplus_{p\ge 0}({\mathrm
J}_X\Lscr)^{\widehat\otimes_X p}
\end{equation}
which sends
\[
\alpha_0\otimes\alpha_1\otimes \cdots\otimes \alpha_p\in
\widehat{\mathrm{HC}}_{X,p}({\mathrm J}_XL)
\]
to
\[
\epsilon(\alpha_0)\alpha_1\otimes \cdots\otimes \alpha_p \in
\mathrm{HC}_{X,p}^\Lscr
\]
The map \eqref{ref-6.2-22} commutes with differentials.
\end{proposition}
\begin{proof} That the restriction of \eqref{ref-6.2-22} is an isomorphism
is proved in \cite[Prop.\ 1.11]{CDH}. That \eqref{ref-6.2-22} commutes with differentials
is an easy verification.
\end{proof}
The cap product of a
section $D=D_1\otimes\cdots\otimes D_p$ of $\mathrm{HC}_{\Lscr,X}^{p}$
and a section $\alpha=\alpha_0\otimes\cdots\otimes \alpha_q$ of
$\widehat{\mathrm{HC}}_{X,q}({\mathrm J}_X \Lscr)$ was in \cite[\S3.4]{CRVdB} defined as
%\footnote{In fact the right action $\alpha\cap D$ was defined
%which differs from our $D\cap \alpha$ by a sign $(-1)^{|D||\alpha|}$. As
%it should.}
\[
D\cap \alpha=\alpha_0{}^2\nabla_{D_1}\alpha_1\cdots {}^2\nabla_{D_p}\alpha_p\otimes
\alpha_{p+1}\otimes\cdots
\otimes\alpha_q
\]
and for $f\in \mathrm{HC}0_{\Lscr,X}=\Oscr_X$:
\[
f\cap \alpha=f\alpha_0\otimes\cdots
\otimes\alpha_q.
\]
One verifies that this cap product is compatible with differentials.
\[
\mathrm b_H(D\cap \alpha)=\mathrm d_H D\cap \alpha+(-1)^{|D|} D\cap \mathrm b_H\alpha
\]
The fact that ${}^G\nabla$ and ${}^2\nabla$ commute yields immediately
\[
{}^G\nabla_l(D\cap \alpha)=D\cap {}^G\nabla_l(\alpha)
\]
Hence $\cap$ descends to a cap product
\begin{equation}
\label{ref-6.3-23}
\cap:\mathrm{HC}_{\Lscr,X}^{\bullet}\times
\mathrm{HC}^\Lscr_{X,\bullet}\r\mathrm{HC}^\Lscr_{X,\bullet}
\end{equation}
compatible with the differentials.
\begin{proposition} For a section $D=D_1\otimes\cdots\otimes D_p$ of
$\mathrm{HC}_{\Lscr,X}^{p}$ and a section $\alpha=\alpha_1\otimes\cdots\otimes \alpha_q$
of
$\mathrm{HC}^{\Lscr}_{X,q}$ (using the identification \eqref{ref-6.1-21})
we have
\begin{equation}
\label{ref-6.4-24}
D\cap\alpha=\alpha_1(D_1)\cdots \alpha_p(D_p)\alpha_{p+1}\otimes\cdots\otimes \alpha_q
\end{equation}
and for $f\in \mathrm{HC}0_{\Lscr,X}=\Oscr_X$:
\[
f\cap \alpha=f\alpha_1\otimes\cdots
\otimes\alpha_q
\]
\end{proposition}
\begin{proof} This is a straightforward verification.
\end{proof}
\section{A digression}
The Hochschild cohomology as we have defined it is computed in
the\def\Dis{\operatorname{Dis}} category $\Mod({\mathrm J}_X\Lscr)$. Inside
$\Mod({\mathrm J}_X\Lscr)$ we have the full subcategory $\Dis({\mathrm J}_X\Lscr)$ of
modules whose sections are locally annihilated by powers of
$\mathrm J^c_X\Lscr$.
\begin{lemma}
\label{new1}
$\Dis({\mathrm J}_X\Lscr)$ is a Grothendieck subcategory of $\Mod(\mathrm{J}_X\Lscr)$.
\end{lemma}
\begin{proof} $\Dis({\mathrm J}_X\Lscr)$ is clearly an abelian
  subcategory of $\Mod(\mathrm{J}_X\Lscr)$ which is closed under
  colimits. Hence it remains to construct a set of generators. The
  objects $j_{!}(\mathrm{J}_U\Lscr_U)/(\mathrm{J}^c_U\Lscr_U)^n$ where $j:U\r X$ runs
through the
  objects of the site and $n$ is arbitrary, do the job.
\end{proof}
Since  $\Oscr_X\in \Dis({\mathrm J}_X\Lscr)$ this suggests the following
alternative definition for Hochschild cohomology
\[
\HH^n_{\Lscr,\mathrm{dis}}(X)=\operatorname{Ext}^n_{\Dis({\mathrm
J}_X\Lscr)}(\mathcal{O}_X,\mathcal{O}_X)\\
\]
We show below that this yields in fact the same result as before. Along
the way we will prove some technical results needed later.

For $\Kscr\in \Dis({\mathrm J}_X\Lscr)$ let $
\uRHom_{\Dis({\mathrm J}_X\Lscr)}(\Kscr,-)
$
be the right derived functor of
$
\uHom_{\Dis({\mathrm J}_X\Lscr)}(\Kscr,-)
$
which sends $\Fscr\in \Dis({\mathrm J}_X\Lscr)$ to the sheaf $U\mapsto\Hom_{{\mathrm
J}_U\Lscr}(\Kscr{|}U,\Fscr{\mid} U)$.
The exactness of $j_!$ implies that injectives in
$\Dis(\mathrm{J}_X\Lscr)$ are preserved under restriction.  This
implies that $ \uRHom_{\Dis({\mathrm J}_X\Lscr)}(\Kscr,-) $ is
compatible with restriction.

\begin{lemma}
\label{ref-7.1-25}
 Let $\Mscr\in \Dis({\mathrm J}_X\Oscr_X)$. The natural map
\begin{equation}
\label{ref-7.1-26}
\uRHom_{\Dis({\mathrm J}_X\Lscr)}(\mathcal{O}_X,\Mscr)\r
\uRHom_{{\mathrm J}_X\Lscr}(\mathcal{O}_X,\Mscr)
\end{equation}
is an isomorphism.
\end{lemma}
\begin{proof}
We may check this locally.
Therefore we may assume that $\Lscr$ is free
over $\Oscr_X$ and ${\mathrm J}_X\Lscr=\Oscr_X[[x_1,\dots,x_d]]$.

Let $E$ be an injective object in $\Dis({\mathrm J}_X\Lscr)$. We need to check that
$\mathcal{E}\mathit{xt}^n_{{\mathrm J}_X\Lscr}(\Oscr_X,E)=0$ for $n>0$.

Let $K_\bullet=\Oscr_X[[x_1,\ldots,x_d]][\xi_1,\ldots,x_d]$ be the Koszul resolution of
$\Oscr_X$ associated to
the regular sequence $(x_1,\ldots,x_d)$ in ${\mathrm J}_X\Lscr$ (with differential
$d\xi_i=x_i$). Then
\[
\uRHom_{{\mathrm J}_X\Lscr}(\Oscr_X,E)=\uHom_{{\mathrm J}_X\Lscr}(K^\bullet,E)
\]
Now put for $p\ge 1$
\[
{}^p K_\bullet=K_\bullet/(x_1,\ldots,x_d,\xi_1,\ldots,\xi_d)^p
\]
Passing to associated graded objects it is easy to see that ${}^p
K_\bullet$ (equipped with the differential inherited from $K_\bullet$)
is a resolution of $\Oscr_X$. Since ${}^p K_\bullet$ is a complex in
$\Dis({\mathrm J}_X\Lscr)$ and $E$ is injective in $\Dis({\mathrm J}_X\Lscr)$ we find
\[
H^n(\uHom_{{\mathrm J}_X\Lscr}({}^p K_\bullet,E))=\begin{cases}
0&\text{$n>0$}\\
\uHom_{{\mathrm J}_X\Lscr}(\Oscr_X,E)&\text{$n=0$}
\end{cases}
\]
We find for $n>0$:
\begin{align*}
H^n(\uHom_{{\mathrm J}_X\Lscr}( K_\bullet,E))&=H^n(\dirlim_p \uHom_{{\mathrm
J}_X\Lscr}({}^p K_\bullet,E))\\
&=\dirlim_p H^n(\uHom_{{\mathrm J}_X\Lscr}({}^p K_\bullet,E))\\
&=0
\end{align*}
The first line is based on the observation that for any $\Mscr\in
\Dis({\mathrm J}_X\Lscr)$ we have
\[
\dirlim_p \uHom_{\mathrm J_X\mathcal L}({\mathrm J}_X\Lscr/({\mathrm
J}_X^c\Lscr)^p,\Mscr)=\Mscr \qed
\]
\def\qed{}\end{proof}
\begin{lemma}
  \label{ref-7.2-27} For any $\mathcal K$, $\mathcal L$ in
  $\Dis(\mathrm J_X\mathcal L)$ there is the following identity
\begin{equation}
\label{ref-7.2-28}
\RHom_{\Dis({\mathrm J}_X\Lscr)}(\Kscr,\Lscr)=R\Gamma(X,\uRHom_{\Dis({\mathrm
J}_X\Lscr)}(\Kscr,\Lscr))
\end{equation}
in $D(\operatorname{Ab})$.
\end{lemma}
\begin{proof}
To check \eqref{ref-7.2-28} we need to verify that if $E$ is an
injective object in $\Dis({\mathrm J}_X\Lscr)$ then
$\Nscr=\uHom_{{\mathrm J}_X\Lscr}(\Kscr,E)$ is acyclic for $\Gamma(X,-)=
\Hom_{\underline{\ZZ}_X}(\underline{\ZZ}_X,-)$.  This is trivial if we are on a space
since one
verifies immediately that $\Nscr$ is flabby. If $X$ is a site then we
can proceed as follows.  By general properties of $\Ext$ an element $\alpha$ of
$\Ext^n_{\underline{\ZZ}_X}(\underline{\ZZ}_X,\Nscr)$ is represented by an element in
$H^n(\Hom_{\underline{\ZZ}_X}(\Gscr^\bullet,\Nscr))$ for some resolution
$\Gscr^\bullet\r \underline{\ZZ}_X\r 0$ in $\Mod(\underline{\ZZ}_X)$ and by resolving
$\Gscr^\bullet$ further we may without loss of
generality assume that $\Gscr^{\bullet}$ is flat.  Then we have
\begin{align*}
H^n(\Hom_{\underline{\ZZ}_X}(\Gscr^\bullet,\Nscr))&=H^n(\Hom_{\underline{\ZZ}_X}(\Gscr^\bullet,\uHom_{{\mathrm
J}_X\Lscr}(\Kscr,E)))\\
&=H^n(\Hom_{{\mathrm J}_X\Lscr}(\Gscr^\bullet\otimes_{\underline{\ZZ}_X}\Kscr,E))
\end{align*}
where ${\mathrm J}_X\Lscr$ acts on the second factor of
$\Gscr^\bullet\otimes_{\underline{\ZZ}_X}\Kscr$.
Since $\Gscr^\bullet\r \underline{\ZZ}_X\r 0$ consist entirely of flat
$\underline{\ZZ}_X$ modules we have
\[
H^n(\Gscr^\bullet\otimes_{\underline{\ZZ}_X}\Kscr)=
\begin{cases}
\Kscr&\text{if $n=0$}\\
0&\text{otherwise}
\end{cases}
\]
Since $\Gscr^\bullet\otimes_{\underline{\ZZ}_X}\Kscr$ is a complex in
$\Dis({\mathrm J}_X\Lscr)$ and $E$ was assumed to be injective in
$\Dis({\mathrm J}_X\Lscr)$ we conclude that for $n>0$
\begin{align*}
H^n(\Hom_{\underline{\ZZ}_X}(\Gscr^\bullet,\Nscr))&=H^n(\Hom_{{\mathrm
J}_X\Lscr}(\Gscr^\bullet\otimes_{\underline{\ZZ}_X}\Kscr,E))\\
&=
\Hom_{{\mathrm J}_X\Lscr}(H^n(\Gscr^\bullet\otimes_{\underline{\ZZ}_X}\Kscr),E)\\
&=0
\end{align*}
Hence $\alpha=0$. Since this holds for any element of
$\Ext^n_{\underline{\ZZ}_X}(\underline{\ZZ}_X,\Nscr)$ we
conclude $\Ext^n_{\underline{\ZZ}_X}(\underline{\ZZ}_X,\Nscr)=0$.
\end{proof}
\begin{proposition} The natural map
\[
\HH^n_{\Lscr,\mathrm{dis}}(X)\r \HH^n_{\Lscr}(X)
\]
is an isomorphism.
\end{proposition}
\begin{proof}
We need to prove that the natural map
\begin{equation}
\label{ref-7.3-29}
\RHom_{\Dis({\mathrm J}_X\Lscr)}(\mathcal{O}_X,\mathcal{O}_X)\r
\RHom_{{\mathrm J}_X\Lscr}(\mathcal{O}_X,\mathcal{O}_X)
\end{equation}
is an isomorphism in $D(\operatorname{Ab})$.

By the local global spectral sequences for $\RHom_{{\mathrm J}_X\Lscr}(-,-)$ and
$\RHom_{\Dis({\mathrm J}_X\Lscr)}(-,-)$ (Lemma \ref{ref-7.2-27}) this reduces to Lemma
\ref{ref-7.1-25}.
\end{proof}
\section{The bar resolution}
The $\Lscr$ bar complex is defined as
\[
\Bscr^{\Lscr}_{X,\bullet}=\bigoplus_{p\geq 0} ({\mathrm J}_X\Lscr)^{\widehat{\otimes}_{X}
p+1}
\]
with differential
\[
\mathrm b'_H(\alpha_0\otimes\cdots\otimes \alpha_p)=
\alpha_0\alpha_1\otimes\cdots\otimes\alpha_p-\alpha_0\otimes\alpha_1\alpha_2
\otimes\cdots\alpha_p+(-1)^p\alpha_0\otimes\cdots\otimes\alpha_{p-1}\epsilon(\alpha_p)
\]
We consider $\Bscr^{\Lscr}_{X,\bullet}$ as a ${\mathrm J}_X\Lscr$-module via
\[
\alpha\cdot (\alpha_0\otimes\cdots\otimes \alpha_p)=\alpha\alpha_0\otimes
\alpha_1\otimes\cdots\otimes\alpha_p
\]
Clearly $\mathrm b'_H$ is ${\mathrm J}_X\Lscr$-linear.
The map $\epsilon:{\mathrm J}_X\Lscr=\Bscr^{\Lscr}_{X,0}\r \Oscr_X$ defines an
${\mathrm J}_X\Lscr$ augmentation for $\Bscr^{\Lscr}_{X,\bullet}$.
\begin{proposition} \label{ref-8.1-30} The bar complex is a resolution of $\Oscr_X$ as
a ${\mathrm J}_X\Lscr$-module.
\end{proposition}
\begin{proof} We need to prove that
\[
\Bscr^{\Lscr}_{X,\bullet}\xrightarrow{\epsilon} \Oscr_X\r 0
\]
is acyclic. To this end it suffices to construct a contracting homotopy as sheaves of
abelian groups.
We do this as follows: we define $h_{-1}:\Oscr_X\r \Bscr^{\Lscr,0}={\mathrm J}_X\Lscr$ as
$h_{-1}(f)=f\cdot 1$ and for $p\ge 0$ we put
\[
h_p(\alpha_0\otimes\cdots \otimes \alpha_p)=1\otimes \alpha_0\otimes\cdots
\otimes \alpha_p
\]
It is easy to verify that this is indeed a contracting homotopy.
\end{proof}
We need a variant on the construction of $\Bscr_{X,p}^{\Lscr}$. Put
${}^q\!{\mathrm J}_X\Lscr={\mathrm J}_X\Lscr/(\mathrm J^c_X\Lscr)^q$. Define
\[
{}^q\Bscr_{X,\bullet}^{\Lscr}=\bigoplus_{p\geq 0} ({}^q\!{\mathrm J}_X\Lscr)^{\otimes_{X}
p+1}
\]
In the same way as in the proof of Proposition \ref{ref-8.1-30} one
proves that ${}^q\Bscr_{X,\bullet}^{\Lscr}$ is a resolution of $\Oscr_X$.
\begin{lemma}
For $\Mscr\in \Dis({\mathrm J}_X\Lscr)$ we have
\begin{equation}
\label{ref-8.1-31}
\begin{aligned}
\uHom^{\text{cont}}_{{\mathrm J}_X\Lscr}(\Bscr_{X,p}^\Lscr,\Mscr)&=\dirlim_q
\uHom_{{\mathrm J}_X\Lscr}({}^q\Bscr_{X,p}^\Lscr,\Mscr)\\
&=\injlim_q\uHom_{\Oscr_X}(({}^q\!{\mathrm J}_X\Lscr)^{\otimes_{X} p},\Mscr)
\end{aligned}
\end{equation}
\end{lemma}
\begin{proof}
With the notations as in \S\ref{ref-3.3-5} we have
\[
\uHom^{\text{cont}}_{{\mathrm J}_X\Lscr}(\Bscr_{X,p}^\Lscr,\Mscr)=
\injlim_n\uHom^{\text{cont}}_{{\mathrm J}_X\Lscr}(\Bscr_{X,p}^\Lscr/F_n\Bscr_{X,p}^\Lscr
,\Mscr)
\]
 where
\[
F_n\Bscr_{X,p}^\Lscr =\sum_{\sum_i n_i=n}
(\mathrm J^c_X\Lscr)^{n_1}\ctimes_X\cdots \ctimes_X
(\mathrm J^c_X\Lscr)^{n_{p+1}}
\]
On the other hand we have
\[
\dirlim_q
\uHom_{{\mathrm J}_X\Lscr}({}^q\Bscr_{X,p}^\Lscr,\Mscr)=
\dirlim_q
\uHom_{{\mathrm J}_X\Lscr}(\Bscr_{X,p}^\Lscr/G_q\Bscr_{X,p}^\Lscr,\Mscr)
\]
with
\[
G_q\Bscr_{X,p}^\Lscr=
\sum_i ({\mathrm J}_X\Lscr)^{\ctimes_X i}\ctimes (\mathrm J^c_X\Lscr)^{q}\ctimes_X
 ({\mathrm J}_X\Lscr)^{\ctimes_X {p-i}}
\]
It now suffices to note that the filtrations
$(F_n\Bscr_{X,p}^\Lscr)_n$ and $(G_q\Bscr_{X,p}^\Lscr)_q$
are cofinal inside $\Bscr_{X,p}^\Lscr$.
\end{proof}
\begin{proposition}
\label{ref-8.3-32}
In $D(\Mod(\Oscr_X))$ we have
\begin{align}
\Oscr_X\Lotimes_{{\mathrm J}_X\Lscr}\Oscr_X&=\Oscr_X\otimes_{{\mathrm
J}_X\Lscr}\Bscr_{X,\bullet}^\Lscr
\label{ref-8.2-33}
\end{align}
Furthermore for $\Mscr\in \Dis({\mathrm J}_X\Lscr)$ the composition
\begin{multline*}
\uHom^{\mathrm{cont}}_{{\mathrm J}_X\Lscr}(\Bscr_{X,\bullet}^\Lscr,\Mscr)
\xrightarrow{\sigma} \uHom_{{\mathrm J}_X\Lscr}(\Bscr_{X,\bullet}^\Lscr,\Mscr)\\
\xrightarrow{\tau} \uRHom_{{\mathrm
J}_X\Lscr}(\Bscr_{X,\bullet}^\Lscr,\Mscr)\overset{\mu}{\cong}
\uRHom_{{\mathrm J}_X\Lscr}(\Oscr_X,\Mscr)
\end{multline*}
(whith $\sigma$, $\tau$, $\mu$ the obvious natural maps) yields an isomorphism
\begin{align}
\uRHom_{{\mathrm
J}_X\Lscr}(\Oscr_X,\Mscr)&\overset{(\mu\tau\sigma)^{-1}}\cong\uHom^{\mathrm{cont}}_{{\mathrm
J}_X\Lscr}(\Bscr_{X,\bullet}^\Lscr,\Mscr)\label{ref-8.3-34}
\end{align}
\end{proposition}
\begin{proof}
  We first discuss \eqref{ref-8.3-34} (see also \cite[Thm 0.3]{ye2}).
  Let $E^\bullet$ be an injective resolution of $\Mscr$ in
  $\Dis({\mathrm J}_X\Lscr)$.  According to Lemma \ref{ref-7.1-25} we
  know that injectives in $\Dis(L_X\Lscr)$ are acyclic for
  $\uHom_{{\mathrm J}_X \Lscr}(\Oscr_X,-)$. Hence
   $\uRHom_{{\mathrm J}_X \Lscr}(\Oscr_X,\Mscr)\cong\uHom_{{\mathrm J}_X
\Lscr}(\Oscr_X,E^\bullet)$.

   Furthermore from the second line of \eqref{ref-8.1-31}, taking into
   account that $({}^q\! {\mathrm J}_X\Lscr)^{\otimes_{X} p}$ is
   locally free over $\Oscr_X$ and that direct limits are exact it
   follows that the cohomology for the columns of the double complex
   $\uHom^{\mathrm{cont}}_{{\mathrm
       J}_X\Lscr}(\Bscr_{X,\bullet}^\Lscr,E^\bullet)$ is equal to
   $\uHom^{\mathrm{cont}}_{{\mathrm
       J}_X\Lscr}(\Bscr_{X,\bullet}^\Lscr,\Mscr)$.  Thus
   $\uHom^{\mathrm{cont}}_{{\mathrm J}_X\Lscr}(\Bscr_{X,\bullet}^\Lscr,\Mscr)\cong\uHom^{\mathrm{cont}}_{{\mathrm
       J}_X\Lscr}(\Bscr_{X,\bullet}^\Lscr,E^\bullet)$ as objects in
   $D(\Mod(\Oscr_X))$.

We claim that the cohomology for the rows of $\uHom^{\mathrm{cont}}_{{\mathrm
       J}_X\Lscr}(\Bscr_{X,\bullet}^\Lscr,E^\bullet)$   is equal to $\uHom_{{\mathrm
J}_X\Lscr}(\Oscr_X, E^\bullet)$. Let $E$ be a single injective in $\Dis({\mathrm
J}_X\Lscr)$.
Standard manipulations with adjoint functors establish that
\[
{}^q\!E=\uHom_{{\mathrm J}_X\Lscr}({\mathrm J}_X\Lscr/({\mathrm J}_X\Lscr)^q,E)
\]
is injective in $\Mod({}^q\!{\mathrm J}_X\Lscr)$. Using the fact that
${}^q\Bscr_{X,p}^\Lscr$ is a resolution of $\Oscr_X$ (as noted above)
we compute
\begin{align*}
H^n(\uHom^{\mathrm{cont}}_{{\mathrm J}_X\Lscr}(\Bscr_{X,p}^\Lscr,{}^q\!E))&=
\dirlim_q H^n(\uHom_{{}^q\!{\mathrm J}_X\Lscr}({}^q\Bscr_{X,p}^\Lscr,{}^q\!E))
\\
&=\dirlim_q \uHom_{{}^q\!{\mathrm J}_X\Lscr}(H_n({}^q\Bscr_{X,p}^\Lscr),{}^q\!E)\\
&=
\begin{cases}
0&\text{$n>0$}\\
\dirlim_q \uHom_{{}^q\!{\mathrm J}_X\Lscr}(\Oscr_X,{}^q\!E)
&\text{$n=0$}
\end{cases}\\
&=
\begin{cases}
0&\text{$n>0$}\\
\uHom_{{\mathrm J}_X\Lscr}(\Oscr_X,E)
&\text{$n=0$}
\end{cases}
\end{align*}
Thus as objects in $D(\Mod(\Oscr_X))$ we have  $\uHom_{{\mathrm J}_X\Lscr}(\Oscr_X,
E^\bullet)\cong\uHom^{\mathrm{cont}}_{{\mathrm
       J}_X\Lscr}(\Bscr_{X,\bullet}^\Lscr,E^\bullet)$.
  We now obtain a commutative diagram
\[
\xymatrix{
\uHom_{{\mathrm J}_X\Lscr}(\Oscr_X,E^\bullet) \ar[r]^{\cong}\ar[d]_{\cong} &
\uHom_{{\mathrm J}_X\Lscr}^{\text{cont}}(\Bscr_{X,\bullet}^\Lscr,E^\bullet) \ar[r]^\sigma
&
\uHom_{{\mathrm J}_X\Lscr}(\Bscr_{X,\bullet}^\Lscr,E^\bullet) \ar[d]^\tau \\
\uRHom_{{\mathrm J}_X\Lscr}(\Oscr_X,E^\bullet)\ar[rr]_{\mu^{-1}}^\cong&& \uRHom_{{\mathrm
J}_X\Lscr}(\Bscr_{X,\bullet}^\Lscr,E^\bullet)
}
\]
where the ``$\cong$'' denote quasi-isomorphisms. It follows that
$\mu\tau\sigma$ is indeed an isomorphism in $D(\Mod(\Oscr_X))$.

\medskip

Now we discuss \eqref{ref-8.2-33}.
It is easy to see that we have to show that
\begin{equation}
\label{ref-8.4-35}
H^n(\Oscr_X\Lotimes_{{\mathrm J}_X\Lscr}\Bscr_{X,p}^\Lscr)=0\qquad\text{for $n>0$}
\end{equation}
We may check this locally. I.e.\ we may assume
${\mathrm J}_X\Lscr=\Oscr_X[[x_1,\ldots,x_d]]$ and hence
\begin{equation}
\label{ref-8.5-36}
\Bscr_{X,p}^{\Lscr}=\Oscr_X[[x_1^{(1)},\ldots,x_d^{(p+1)}]]
\end{equation}
with ${\mathrm J}_X\Lscr$ acting through the variables $x_1^{(1)},\ldots,x_d^{(1)}$.

Then $\Oscr_X={\mathrm J}_X\Lscr/(x_1,\ldots,x_d)$ and
\eqref{ref-8.5-36} shows that  $(x_1,\ldots,x_d)$ forms
a regular sequence on $\Bscr_{X,p}^{\Lscr}$. The required vanishing
in \eqref{ref-8.4-35} now follows in the usual way.
\end{proof}
\section{Discussion of the cup product.}
We will consider $D(\Mod(\Oscr_X))$ as a symmetric monoidal category through
the derived tensor product over $\Oscr_X$.
\begin{proposition}
\label{ref-9.1-37}
There is a canonical isomorphism of algebra objects  in $D(\Mod(\Oscr_X))$
\begin{equation}
\label{ref-9.1-38}
\Phi:(\mathrm{HC}^{\bullet}_{\Lscr,X})^{\text{opp}}\overset{\cong}{\longrightarrow}
\uRHom_{{\mathrm J}_X\Lscr}(\Oscr_X,\Oscr_X)
\end{equation}
which sends the opposite of the cup product to the Yoneda product.
\end{proposition}
\begin{proof} We have
\[
\uRHom_{{\mathrm J}_X\Lscr}(\Oscr_X,\Oscr_X)=\uRHom_{{\mathrm
J}_X\Lscr}(\Bscr_{X,\Lscr}^\bullet,\Bscr_{X,\Lscr}^\bullet)
\]
Thus $\Phi$ is an element of
\[
\Hom_{\Oscr_X}\biggl((\mathrm{HC}^{\bullet}_{\Lscr,X})^{\text{opp}},\uRHom_{{\mathrm
J}_X\Lscr}(\Bscr^\bullet_{X,\Lscr},\Bscr^\bullet_{X,\Lscr})\biggr)
\]
or using the $\Hom$-tensor relations (see \S\ref{new7}), a map in
$D(\Mod(\mathrm{J}_X\Lscr))$
\[
(\mathrm{HC}^{\bullet}_{\Lscr,X})^{\text{opp}}\otimes_{\Oscr_X} \Bscr^\bullet_{X,\Lscr}\r
\Bscr^\bullet_{X,\Lscr}
\]
Note that the tensor product is not derived since both factors are $\Oscr_X$-flat.

Thus to define a morphism like in \eqref{ref-9.1-38} it suffices to
define a $J_X\Lscr$-linear action of $\Bscr_{X,\Lscr}^\bullet$ on
$\left(\mathrm{HC}^{\bullet}_{\Lscr,X}\right)^{\text{opp}}$. One easily verifies that
if the action makes $\Bscr_{X,\Lscr}^\bullet$ into a
DG-module over $\left(\mathrm{HC}^{\bullet}_{\Lscr,X}\right)^{\text{opp}}$
then
$\Phi$ is an algebra morphism.

For a section $D=D_1\otimes\cdots \otimes D_p$ of $\mathrm{HC}^p_{\Lscr,X}$ and
a section $\alpha=\alpha_0\otimes\cdots\otimes\alpha_q$ of $\Bscr_{X,q}^\Lscr$ we put
\begin{equation}
\label{ref-9.2-39}
D\cap \alpha
=
\begin{cases}
\alpha_0{}^2\nabla_{D_1}\alpha_1\cdots {}^2\nabla_{D_p}\alpha_p\otimes
\alpha_{p+1}\otimes\cdots \otimes \alpha_q& \text{if $q\ge p$}\\
0&\text{otherwise}
\end{cases}
\end{equation}
and for $f$ a section of $\mathrm{HC}0_{\Lscr,X}=\Oscr_X$ we put
\[
f\cap (\alpha_0\otimes\cdots\otimes\alpha_q)=f\alpha_0\otimes\cdots\otimes\alpha_q
\]
This action is obviously $\mathrm{J}_X\Lscr$-linear and furthermore it is an easy
verification that
\[
\mathrm{b}'_H(D\cap \alpha)=\mathrm{d}_H(D)\cap \alpha+(-1)^{|D|}
D\cap \mathrm{b}'_H(\alpha)
\]
Hence we have indeed defined a morphism as in \eqref{ref-9.1-38}.
The fact that it sends the opposite of the cup product to the Yoneda
product follows from the easily verified
identity:
\[
(D\cup E)\cap \alpha=(-1)^{|D||E|}E\cap(D\cap \alpha)
\]
%It remains to prove that \eqref{ref-9.1-38} is an isomorphism. This can be done as
%complexes so we may forget about the DG-structures.
It is easy to see that the composition
\begin{equation}
\label{ref-9.3-40}
\mathrm{HC}_{\Lscr,X}^\bullet\r
\uHom_{{\mathrm J}_X\Lscr}(\Bscr_{X,\bullet}^\Lscr,\Bscr_{X,\bullet}^\Lscr)
\xrightarrow{\epsilon \circ-} \uHom_{{\mathrm J}_X\Lscr}(\Bscr_{X,\bullet}^\Lscr,\Oscr_X)
\end{equation}
is given by (we will pass silently over the special case $p=0$  as it is easy)
\[
D_1\otimes\cdots \otimes D_p\mapsto \left(\alpha_0\otimes\cdots\otimes
\alpha_q\mapsto
\begin{cases}
\langle\alpha_0,1\rangle\langle\alpha_1,D_1\rangle\cdots\langle\alpha_p,D_p\rangle&\text{if
$p=q$}\\
0&\text{otherwise}
\end{cases}
\right)
\]
From this formula it is easy to see that the image of
$\mathrm{HC}^p_{\Lscr,X}$ under \eqref{ref-9.3-40} lies in
\[
\uHom^{\text{cont}}_{{\mathrm
J}_X\Lscr}(\Bscr^{\Lscr}_{X,p},\Oscr_X)=\uHom_{\Oscr_X}^{\text{cont}}(
({\mathrm J}_X\Lscr)^{\ctimes p},\Oscr_X)
\]
and the resulting map
\[
\mathrm{HC}^p_{\Lscr,X}\r \uHom_{\Oscr_X}^{\text{cont}}(
({\mathrm J}_X\Lscr)^{\ctimes p},\Oscr_X)
\]
is given by
\[
D_1\otimes\cdots \otimes D_p\mapsto (\alpha_1\otimes\cdots\otimes
\alpha_p\mapsto
\langle\alpha_1,D_1\rangle\cdots\langle\alpha_p,D_p\rangle
)\]
which by the discussion in \S\ref{ref-3.3-5} is an isomorphism.

We may now construct the following commutative
diagram %(each of the two squares is commutative on the nose)
\begin{equation}
\label{new2}
\xymatrix{%
\mathrm{HC}_{\Lscr,X}\ar[r]^-{\cap}
\ar[d]_{\cong}& \uHom_{{\mathrm
J}_X\Lscr}(\Bscr_{X,\bullet}^\Lscr,\Bscr_{X,\bullet}^\Lscr)\ar[r]\ar[d] &\uRHom_{{\mathrm
J}_X\Lscr}(\Bscr_{X,\bullet}^{\Lscr},\Bscr_{X,\bullet}^{\Lscr})\ar[d]^\cong\\
\uHom^{\text{cont}}_{\mathrm{J}_X\Lscr}(\Bscr^{\Lscr}_{X,\bullet},\Oscr_X)\ar@(rd,dl)[rr]_{\cong}\ar[r]_\sigma&
 \uHom_{{\mathrm J}_X\Lscr}(\Bscr_{X,\bullet}^\Lscr,\Oscr_X)\ar[r]_\tau&
\uRHom_{{\mathrm J}_X\Lscr}(\Bscr^{\Lscr}_{X,\bullet},\Oscr_X)
}
\vspace*{.5cm}
\end{equation}
Here the left square is commutative by the fact that
\eqref{ref-9.3-40} has its image inside
$\uHom^{\text{cont}}(\Bscr^{\Lscr}_{X,\bullet},\Oscr_X)$ (as we have
discussed in the previous paragraph).   The right horizontal arrows are derived
from the obvious $\mathrm{J}_X\Lscr$-linear actions
\[
\uHom_{{\mathrm J}_X\Lscr}(\Bscr_{X,\bullet}^\Lscr,\Bscr_{X,\bullet}^\Lscr)
\otimes_{\Oscr_X} \Bscr_{X,\bullet}^{\Lscr}\r \Bscr_{X,\bullet}^{\Lscr}
\]
\[
\uHom_{{\mathrm J}_X\Lscr}(\Bscr_{X,\bullet}^\Lscr,\Oscr_X)
\otimes_{\Oscr_X} \Bscr_{X,\bullet}^{\Lscr}\r \Oscr_X
\]
(as explained in the beginning of the proof for the $\mathrm{HC}_{\Lscr,X}$-action)

\medskip

The curved arrow is an isomorphism by Proposition \ref{ref-8.3-32}.
It follows that \eqref{ref-9.1-38} is indeed an isomorphism.
\end{proof}
The following result was observed by the referee.
\begin{proposition} \label{new6} The map of DG-algebras
\[
(\mathrm{HC}^{\bullet}_{\Lscr,X})^{\text{opp}}\overset{\cong}{\longrightarrow}
\uHom_{{\mathrm
    J}_X\Lscr}^{\text{cont}}(\Bscr_{X,\bullet}^\mathcal L,\Bscr_{X,\bullet}^\mathcal L)
\]
obtained from the (continuous!) action of
$\mathrm{HC}^{\bullet}_{\Lscr,X}$ on $\Bscr_{X,\bullet}$ through the
capproduct (see \eqref{ref-9.2-39}), is a quasi-isomorphism.
\end{proposition}
\begin{proof}
The proof is easy in principle but we have to be careful with taking
products of sheaves. In particular the functor $\uHom_{{\mathrm
J}_X\Lscr}^{\text{cont}}(\Bscr_{X,\bullet}^\mathcal L,-)$ is not exact.
This problem is solved by working on the presheaf level.

Looking at the leftmost square of \eqref{new2}
it is clearly sufficient to show that the map
\begin{equation}
\label{new3}
\Tot(\uHom_{{\mathrm
    J}_X\Lscr}^{\text{cont}}(\Bscr^\mathcal L_{X,\bullet},\Bscr^\mathcal L_{X,\bullet}))
\r
\uHom_{{\mathrm
    J}_X\Lscr}^{\text{cont}}(\Bscr^\mathcal L_{X,\bullet},\Oscr_X)
\end{equation}
obtained from the augmentation is a quasi-isomorphism.
In this framework, we regard $\uHom_{{\mathrm
    J}_X\Lscr}^{\text{cont}}(\Bscr^\mathcal L_{X,\bullet},\Bscr^\mathcal L_{X,\bullet})$
as
a double complex located in the second quadrant.
Thus the horizontal differential comes from the differential on the rightmost copy of
$\Bscr_{X,\bullet}$.

We first replace our site with a new one $X'$ containing only the
objects $U$ for which $\Lscr_U$ is free. Obviously $X$ and $X'$ have
the same sheaf theory.
Let $U$ be an object of $X'$. We will show that
\begin{multline}
\label{new4}
\Tot(\uHom_{{\mathrm
    J}_X\Lscr}^{\text{cont}}(\Bscr^\mathcal L_{X,\bullet},\Bscr^\mathcal
L_{X,\bullet}))(U)=\Tot(\Hom_{{\mathrm
    J}_U\Lscr}^{\text{cont}}(\Bscr^\mathcal L_{U,\bullet},\Bscr^\mathcal L_{U,\bullet}))\\
\r
\Hom_{{\mathrm
    J}_U\Lscr}^{\text{cont}}(\Bscr^\mathcal L_{U,\bullet},\Oscr_U)=\uHom_{{\mathrm
    J}_X\Lscr}^{\text{cont}}(\Bscr^\mathcal L_{X,\bullet},\Oscr_X)(U)
\end{multline}
is a quasi-isomorphism. Thus we obtain a presheaf version of the
required quasi-isomorphism. We finish by applying
sheafification.

Some diagram chasing reveals that to prove that \eqref{new4} is a quasi-isomorphism it is
sufficient to check that the rows of the double complex $\Hom_{{\mathrm
    J}_U\Lscr}^{\text{cont}}(\Bscr^\mathcal L_{U,\bullet},\Bscr^\mathcal L_{U,\bullet})$
have the correct
cohomology. I.e.\ for any $n$ we must check that the map
\[
\Hom_{{\mathrm
    J}_U\Lscr}^{\text{cont}}(\Bscr^\mathcal L_{U,n},\Bscr^\mathcal L_{U,\bullet})\r
\Hom_{{\mathrm
    J}_U\Lscr}^{\text{cont}}(\Bscr^\mathcal L_{U,n},\Oscr_{U})
\]
is a quasi-isomorphism.

Now the local form of ${\mathrm
    J}_U\Lscr$  (see \eqref{new5}) implies that $\Bscr^\mathcal L_{U,n}$ is topologically
free. Denote the indexing set for a basis by $I$. Then the functor
\[
\Hom_{{\mathrm
    J}_U\Lscr}^{\text{cont}}(\Bscr^\mathcal L_{U,n},-)
\]
sends a sheaf $\Mscr$ of complete linear topological $\mathrm{J}_U\Lscr$-modules to
$\Mscr(U)^I$. Hence
it remains to show that
\[
\Bscr^\mathcal L_{U,\bullet}(U)\r \Oscr_U(U)\r 0
\]
is acyclic (since then we may invoke exactness for products of abelian groups).

The fact that $\Gamma(U,-)$ commutes with inverse limits and hence with completions
implies that
\[
\Bscr^\mathcal L_{U,\bullet}(U)=\Bscr^\mathcal L_{\bullet}(\Oscr(U))
\]
where
\[
\Bscr^\mathcal L_{\bullet}(\Oscr(U))=\bigoplus_{p\geq 0} ({\mathrm
J}_U\Lscr(U))^{\widehat{\otimes}_{X} p+1}
\]
with the usual differential and
\[
{\mathrm J}_U\Lscr(U)\cong\Oscr_U(U)[[x_1,\ldots,x_d]]
\]
 To finish the proof one uses the same method as in the proof of Proposition
\ref{ref-8.1-30} to show that $\Bscr^\mathcal L_{U,\bullet}(U)$ is quasi-isomorphic to
$\Oscr_U(U)$.
\end{proof}
\section{Discussion of the cap product}
Now we prove the following result.
\begin{proposition}
\label{ref-10.1-41} There is a canonical isomorphism in $D(\Mod(\Oscr_X))$
\begin{equation}
\label{ref-10.1-42}
\Psi: \Oscr_X\Lotimes_{{\mathrm J}_X\Lscr}\Oscr_X \overset{\cong}{\longrightarrow}
\mathrm{HC}^{\Lscr}_{X,\bullet}
\end{equation}
which is compatible with $\Phi$ (see \eqref{ref-9.1-38}) in the following
sense: denote the action of $\uRHom_{{\mathrm J}_X\Lscr}(\Oscr_X,\Oscr_X)$ on the
second argument of $\Oscr_X\Lotimes_{{\mathrm J}_X\Lscr}\Oscr_X$ by ``$\cap$''; then we
have
\begin{equation}
\label{ref-10.2-43}
D\cap \Psi(u)=\Phi(D)\cap u
\end{equation}
for a section $D$ of $\mathrm{HC}^{\bullet}_{\Lscr,X}$ and $u$ of $
\Oscr_X\Lotimes_{{\mathrm J}_X\Lscr}\Oscr_X $.
\end{proposition}
\begin{proof}
By \eqref{ref-8.2-33} we have
\[
\Oscr_X\Lotimes_{{\mathrm J}_X\Lscr}\Oscr_X=\Oscr_X\otimes_{{\mathrm J}_X\Lscr}
\Bscr_{X,\bullet}^\Lscr
\]
We now define
\[
\Psi:\Oscr_X\otimes_{{\mathrm J}_X\Lscr}
\Bscr_{X,\bullet}^\Lscr\r \mathrm{HC}_{X,\bullet}^\Lscr
:f\otimes \alpha_0\otimes\cdots\otimes \alpha_p\mapsto
f\epsilon(\alpha_0)\alpha_1\otimes\cdots\otimes \alpha_p
\]
where we use the version of $\mathrm{HC}_{X,\bullet}^\Lscr$ given by Proposition
\ref{ref-6.1-20}.

It is easy to see that $\Psi$ commutes with differentials and is
an isomorphism of complexes. This gives
the required isomorphism in \eqref{ref-10.1-42}.

To verify \eqref{ref-10.2-43} we need to check that the following
diagram is commutative
\[
\xymatrix{
\Oscr_X\otimes_{{\mathrm J}_X\Lscr} \Bscr_{X,\bullet}^{\Lscr}\ar[d]_{1\otimes (D\cap-)}
\ar[r]^-{\Psi}
&\mathrm{HC}_{X,\bullet}^\Lscr\ar[d]^{D\cap-}\\
\Oscr_X\otimes_{{\mathrm J}_X\Lscr} \Bscr_{X,\bullet}^{\Lscr}
\ar[r]_-{\Psi}
&\mathrm{HC}_{X,\bullet}^\Lscr
}
\]
where the cap product formul\ae\  are \eqref{ref-9.2-39} and \eqref{ref-6.4-24}.
This is again a simple verification.
\end{proof}
\section{Main result}
The following is our main result.
\begin{theorem} There are isomorphisms
\label{ref-11.1-44}
\begin{align*}
\Phi:R^n\Gamma(X,\mathrm{HC}_{\Lscr,X}^\bullet)&\overset{\cong}{\longrightarrow}
\HH^n_{\Lscr}(X)\\
\Psi:\HH_n^{\Lscr}(X)&\overset{\cong}{\longrightarrow}
R^{-n}\Gamma(X,\mathrm{HC}_{X,\bullet}^\Lscr)
\end{align*}
such that $(\Phi,\Psi^{-1})$ defines an isomorphism
\[
\bigl(R^\bullet\Gamma(X,\mathrm{HC}_{\Lscr,X,}^\bullet),R^{-\bullet}\Gamma(X,\mathrm{HC}_{X,\bullet}^\Lscr)\bigr)
\cong \bigl(\HH^\bullet_{\Lscr}(X),\HH_\bullet^{\Lscr}(X)\bigr)
\]
compatible with the natural algebra and module structures.
\end{theorem}
\begin{proof} Combining Propositions \ref{ref-9.1-37} and \ref{ref-10.1-41} with the
discussions and the results of Sections 7 and 8 we
  get the result except that
  $R^n\Gamma(X,\mathrm{HC}_{\Lscr,X}^\bullet)$ is replaced by its
  opposite. However $\mathrm{HC}_{\Lscr,X}^\bullet$ is commutative as algebra object in
  $D(\Mod(\underline{k}_X))$. Hence
  $R^\bullet\Gamma(X,\mathrm{HC}_{\Lscr,X}^\bullet)$ is commutative as well.
\end{proof}

\appendix

\section{Jet bundles as formal exponentiations of Lie algebroids}
\label{ref-A-45}

\subsection{Introduction}

This appendix can be read more or or less independently of the main
paper. We show that the jet bundle of a Lie algebroid is a formal
groupoid (see \S\ref{ref-A.2-46}).  For simplicity of notation we work over rings.  Thus
$L$ is
a Lie algebroid locally free of rank $d$ over a commutative
$k$-algebra $R$.  This is not a restriction as we may easily pass to
spaces by sheafification.

We use self explanatory variants of our
earlier notations. E.g.  $\mathrm U_R L$ and ${\mathrm J}_R L$ instead of
$\mathrm U_X\Lscr$ and ${\mathrm J}_X\Lscr$.

The main result of this appendix appears without proof in
\cite[(A.5.10)]{Kapranov3}. At the time this paper was about to be published, Hessel
Posthuma pointed out to us
the anterior paper \cite{KP}, where a different proof appears of the fact that the jet bundle of a Lie
algebroid is a formal groupoid.
We make the relation between our proof and theirs precise in Remark \ref{ref-A.10-59}.

\subsection{Statement of the main result}\label{ref-A.2-46}

We will prove that a number of structures exist on ${\mathrm J}_RL$
(\emph{some of which already appeared before}). All algebras and
morphisms are unitary.
\begin{enumerate}
\item A commutative, associative algebra structure on ${\mathrm J}_RL$ (\emph{as in the
  main paper}).
\item Two ``unit maps''
\begin{align*}
1\!\!1_1:R\r {\mathrm J}_RL\\
1\!\!1_2:R\r {\mathrm J}_RL
\end{align*}
(\emph{with $1\!\!1_1$ being the $R$-algebra structure on ${\mathrm J}_R L$ appearing in
the main paper}).
The unit maps are algebra morphisms.
\item A ``comultiplication''
\[
\Delta:{\mathrm J}_RL\r {\mathrm J}_RL\ctimes_R {\mathrm J}_RL
\]
which is an algebra morphism and also a morphism
of $R$-$R$-bimodules
where $R$
acts through $1\!\!1_1$ on the left of ${\mathrm J}_RL$ and through $1\!\!1_2$ on the
right of ${\mathrm J}_RL$. This convention is also used to interpret the tensor
product ${\mathrm J}_RL\ctimes_R {\mathrm J}_RL$. \emph{Note that this convention is
different from the one which was in use in the main paper.}
\item A ``counit'' (\emph{as in the main paper})
\[
\epsilon:{\mathrm J}_RL\r R
\]
which is an algebra morphism and an $R$-$R$-bimodule morphism where $R$
is considered an $R$-bimodule in the obvious way.
\item An invertible ``antipode''  which is an algebra morphism
\[
S:{\mathrm J}_RL\r {\mathrm J}_RL
\]
and which exchanges the $R$-actions on ${\mathrm J}_RL$ through
$1\!\!1_1$ and $1\!\!1_2$.
\end{enumerate}
These structures satisfy the following additional properties
\begin{enumerate}
\item $\Delta$ is coassociative in the obvious sense.
\item $\epsilon\circ 1\!\!1_1=\Id_R=\epsilon\circ 1\!\!1_1$.
\item For all $\alpha\in {\mathrm J}_RL$ we have
\[
\sum_\alpha (1\!\! 1_1\circ\epsilon)(\alpha_{(1)})\alpha_{(2)}=\alpha
=\sum_\alpha \alpha_{(1)}(1\!\!1_2\circ\epsilon)(\alpha_{(2)})
\]
\item For all all $\alpha\in {\mathrm J}_RL$ we have
\begin{align*}
\sum_\alpha S(\alpha_{(1)})\alpha_{(2)}&=(1\!\!1_2\circ \epsilon)(\alpha)\\
\sum_\alpha \alpha_{(1)}S(\alpha_{(2)})&=(1\!\!1_1\circ \epsilon)(\alpha)
\end{align*}
\end{enumerate}
We will also show
\begin{enumerate} \setcounter{enumi}{4}
\item
 $S^2=\Id_{{\mathrm J}_R L}$
\end{enumerate}
Just as in the Hopf algebra case this turns out to be a formal
consequence of the commutativity of ${\mathrm J}_RL$ \cite[Cor.\ 1.5.12]{Mont1}.
\begin{remark} The listed properties are precisely
those enjoyed by the coordinate ring of a groupoid.
\end{remark}
\begin{remark}
\label{ref-A.2-47} If $R$ is a finitely generated and smooth over a field
  $k$ and $L=T\overset{\text{def}}{=} \Der_k(R)$ then ${\mathrm J}_R L$ is the
  completion of $R\otimes_k R$ at the kernel of the multiplication
  map $R\otimes_k R\r R$. In this case the structure maps are given by
the following formul\ae\
\begin{align*}
1\!\!1_1(a)&=a\ctimes 1\\
1\!\!1_2(a)&=1\ctimes a\\
\Delta(a\ctimes b)&=(a\ctimes 1)\ctimes (1\ctimes b)\\
\epsilon(a\ctimes b)&=ab\\
S(a\ctimes b)&=b\ctimes a
\end{align*}
One easily verifies that these maps have the indicated properties.
\end{remark}
\subsection{Proofs}
The algebra structure on ${\mathrm J}_RL$ and the counit $\epsilon$ were already
introduced in the main paper. See \eqref{ref-3.6-8} and
\eqref{ref-4.1-16}.  We also introduced two commuting left $\mathrm
U_R L$-module structures on ${\mathrm J}_R L$. Namely ${}^G\nabla$ and
${}^2\nabla$ (see \S\ref{ref-3.4-12}). For consistency we will denote
${}^G\nabla$ here by ${}^1\nabla$.
\begin{lemma}
\label{ref-A.3-48} The two actions
  ${}^i\nabla$ are compatible with the natural filtration on
  ${\mathrm J}_RL$. On the associated graded algebra of ${\mathrm J}_RL$, which is equal
to
  $S_R L^\ast$, the actions for ${}^1\nabla$ and ${}^2\nabla$ are as follows
\begin{enumerate}
\item For $r\in R$, ${}^1\nabla_r$ and ${}^2\nabla_r$ are
multiplication by $r$.
\item For $l\in L$, ${}^2\nabla_l$  is the contraction by $l$
and ${}^1\nabla_l$  is the contraction by $-l$.
\end{enumerate}
\end{lemma}
For $r\in R\subset \mathrm U_RL$ we define
$
1\!\!1_i(r)={}^i\nabla_r(1)
$.
Concretely
$1\!\!1_1(r)(D)=rD(1)$ and
$1\!\!1_2(r)(D)=D(r)$ and hence in particular
\begin{equation}
\label{unit}
1\!\!1_1(1)=1=1\!\!1_2(1).
\end{equation}
Here the ``$1$'' in the middle is the \emph{algebra unit} for $\mathrm J_R L$
(see \S\ref{ref-3.3-5}).
Through the identification $\mathrm J_R L=\Hom_R(\mathrm U_R L,R)$
it corresponds to the counit on $\mathrm U_R L$ which sends $D$ to $D(1)$.
Equation \eqref{unit} expresses the fact that $1\!\!1_1$ and $1\!\!1_2$ preserve
algebra units.

%(taking into account that the unit for ${\mathrm J}_R L$ is the counit for
%$\mathrm U_RL$).

We must establish a number of trivial properties of ${}^i\nabla$.
\begin{lemma}
\label{ref-A.4-49}
We have for $\alpha\in {\mathrm J}_RL$, $r,s\in R$ and $i=1,2$
\[
{}^i\nabla_r \alpha=1\!\!1_i(r)\alpha
\]
\[
1\!\!1_i(rs)=1\!\!1_i(r)1\!\!1_i(s)
\]
Thus the maps $1\!\!1_i$ are algebra morphisms $R\r {\mathrm J}_RL$. Furthermore we have
\[
\epsilon \circ 1\!\!1_1=\Id_R=\epsilon \circ 1\!\!1_2
\]
\end{lemma}
\begin{proof}
Assuming the first claim the second claim follows:
\[
1\!\!1_i(rs)\alpha={}^i\nabla_{rs}\alpha
={}^i\nabla_{r}{}^i\nabla_{s}\alpha
=1\!\!1_i(r)1\!\!1_i(s)\alpha
\]
Taking $\alpha=1$ yields what we want.

Now we prove the first claim.
We first consider the case $i=1$.
Let $D\in \mathrm U_R L$. Then we compute
\begin{multline*}
(1\!\!1_1(r)\alpha)(D)=\sum_D (1\!\!1_1(r))(D_{(1)}) \alpha(D_{(2)})
=r\epsilon(D_{(1)})\alpha(D_{(2)})\\
=r\alpha(\epsilon(D_{(1)})D_{(2)})
=r\alpha(D)
=(r\alpha)(D)
=({}^1\nabla_r\alpha)(D)
\end{multline*}
Now we consider the case $i=2$. We compute
\begin{multline*}
(1\!\!1_2(r)\alpha)(D)=\sum_D (1\!\!1_2(r))(D_{(1)}) \alpha(D_{(2)})
=\sum_D D_{(1)}(r) \alpha(D_{(2)})\\
=\sum_D  \alpha(D_{(1)}(r) D_{(2)})
=\alpha(Dr)
=({}^2\nabla_r\alpha)(D)
\end{multline*}
where in the fourth equality we have used the fact that $\mathrm U_RL$ is a so-called
Hopf algebroid with anchor \cite{Xu}.
The third claim is a trivial verification.
\end{proof}
\begin{lemma}
\label{ref-A.5-50} We have
\[
{}^i\nabla_D (1)=(1\!\! 1_i\circ \epsilon)(D)
\]
\end{lemma}
\begin{proof} The right-hand side is equal to ${}^i\nabla_{D(1)}(1)$.
Hence replacing $D$ by $D-D(1)$ we must prove that if $D$ is such that $D(1)=0$ then
${}^i\nabla_D (1)=0$. Such a $D$ is of the form $D'l$, $l\in L$. Hence
we reduce to the case $D=l$. We now conclude by using the explicit
formul\ae\  for ${}^i\nabla$.
\end{proof}
We define two pairings
between $\mathrm U_RL$ and $\alpha\in {\mathrm J}_RL$:
\[
\langle \alpha,D\rangle_i=\epsilon({}^i\nabla_D\alpha)
\]
for $i=1,2$. We have
$
\langle \alpha,D\rangle_2=\alpha(D)
$.
Hence $\langle-,-\rangle_2$ is the pairing $\langle-,-\rangle$ in
the main paper (see \S\ref{ref-3.5-7}).
These pairings satisfy suitable linearity properties
with respect to the $R$-actions via $1\!\!1_i$.
\begin{lemma}
\label{ref-A.6-51}
For $r\in R$, $\alpha \in {\mathrm J}_R L$, $D\in \mathrm U_R L$
we have
\begin{equation}
\label{ref-A.1-52}
\langle \alpha,rD\rangle_i=r\langle\alpha, D\rangle_i=
\langle \alpha 1\!\!1_{\bar{\imath}}(r) ,  D\rangle_i
\end{equation}
\begin{equation}
\label{ref-A.2-53}
\langle \alpha,Dr\rangle_i=\langle 1\!\!1_{i}(r)\alpha,
D\rangle_i
\end{equation}
where $\bar{\imath}=3-i$ and where in the second line we have
used the right action of $R$ on $\mathrm U_R L$ obtained from
the inclusion $R\subset \mathrm U_R L$.
\end{lemma}
\begin{proof}
The identities in \eqref{ref-A.1-52} are a direct consequence
of Lemma \ref{ref-A.4-49}. For \eqref{ref-A.2-53} we compute
\[
\langle \alpha,Dr\rangle_i=\epsilon({}^i\nabla_{Dr}(\alpha))=
\epsilon({}^i\nabla_{D}{}^i\nabla_{r}(\alpha))
=\epsilon({}^i\nabla_{D}(1\!\!1_i(r)\alpha))=\langle 1\!\!1_{i}(r)\alpha,
D\rangle_i\qed
\]
\def\qed{}\end{proof}
 Furthermore we have the following
properties
\begin{lemma} \label{ref-A.7-54} We have for $D\in \mathrm U_RL$,
$\alpha\in{\mathrm J}_R L$:
\begin{equation}
\label{ref-A.3-55}
\langle 1,D\rangle_1=\epsilon(D)=\langle 1,D\rangle_2
\end{equation}
\begin{equation}
\label{ref-A.4-56}
\langle \alpha,1\rangle_1=\epsilon(\alpha)=\langle \alpha,1\rangle_2
\end{equation}
\end{lemma}
\begin{proof} For \eqref{ref-A.3-55} we need to prove
$
({}^i\nabla_D(1))(1)=D(1)
$.
For $i=2$ this is immediate. The case $i=1$ follows by writing out $D$ as a product of
elements of $L$ and
working out the expression ${}^1\nabla_D(1)(1)$. \eqref{ref-A.4-56} is
an immediate verification.
\end{proof}
\begin{lemma}
\label{ref-A.8-57}
The pairings $\langle-,-\rangle_i$ are non-degenerate
$R$-linear pairings (in the sense of Lemma \ref{ref-3.1-10}) where $R$ acts on ${\mathrm
J}_RL$ via $1\!\!1_{\bar{\imath}}$.
\end{lemma}
\begin{proof} The case $i=2$ is Lemma \ref{ref-3.1-10}. The
case $i=1$ is handled in a similar way by passing to associated
graded objects and applying Lemma \ref{ref-A.3-48} to the definition
of $\langle-,-\rangle_1$.
\end{proof}
\begin{lemma} \label{ref-A.9-58} We have
for $D\in \mathrm U_R L$
\[
{}^i\nabla_D(\alpha\beta)=\sum_D {}^i\nabla_{D^{(1)}}(\alpha){}^i\nabla_{D^{(2)}}(\beta)
\]
\end{lemma}
\begin{proof}
  The case $i=1$ we have already encountered in the main paper. It
  expresses the fact that ${\mathrm J}_RL$ is an $R$-algebra (via $1\!\!1_1$)
  and that the multiplication on ${\mathrm J}_R L$ is compatible with the
  Grothendieck connection. See \S\ref{ref-3.3-5} and
  \eqref{ref-3.10-13}.

The case $i=2$ is an easy verification
\begin{multline*}
{}^2 \nabla_D(\alpha\beta)(E)=(\alpha\beta)(ED)
=\sum_{E,D} \alpha(E_{(1)}D_{(1)})\beta(E_{(2)}D_{(2)})\\
={}^2 \nabla_{D_{(1)}}(\alpha)(E_{(1)}){}^2 \nabla_{D_{(2)}}(\alpha)(E_{(2)})
=({}^2 \nabla_{D_{(1)}}(\alpha){}^2 \nabla_{D_{(2)}}(\alpha))(E)\qed
\end{multline*}
\def\qed{}\end{proof}
We define the coproduct on ${\mathrm J}_RL$ through the following
formula
\[
\epsilon({}^1\nabla_D{}^2\nabla_E(\alpha))=\sum_\alpha\langle \alpha_{(1)},D\rangle_1
\langle \alpha_{(2)},E\rangle_2
\]
for all $D,E\in \mathrm U_RL$, $\alpha\in \mathrm J_RL$. The non-degeneracy of the
pairings $\langle-,-\rangle_i$, $i=1,2$ (see Lemma
\ref{ref-A.8-57}), implies that this formula yields indeed
a well-defined element $\sum_\alpha \alpha_{(1)}\ctimes \alpha_{(2)}\in
{\mathrm J}_RL\ctimes_R {\mathrm J}_RL$.
\begin{remark} \label{ref-A.10-59}
Keeping the previous notation, let us recall the simpler expression of \cite{KP} for the
coproduct:
\begin{equation}\label{coprod-KP}
\alpha(DE)=\sum_\alpha\alpha_{(1)}(D\alpha_{(2)}(E))\,.
\end{equation}
Without going into the details (for which we refer to \cite{KP} an references therein), let us also mention that in \cite{KP} the authors consider
a so-called ``translation map" $D\mapsto\sum_DD_+\otimes D_-$
which simplifies considerably the formula for the Grothendieck connection, i.e.\
${}^1\nabla_D(\alpha)(E)=\sum_DD_+(\alpha(D_-E))$.
Using this, our definition for the coproduct reads:
\begin{equation}\label{coprod-ours}
\sum_DD_+(\alpha(D_- E))=\sum_D\sum_\alpha D_+(\alpha_{(1)}(D_-))\alpha_{(2)}(E)\,.
\end{equation}
We now prove that the two definitions actually coincide.
Suppose that \eqref{coprod-KP} is satisfied, then
\begin{eqnarray*}
\sum_DD_+(\alpha(D_- E)) & = & \sum_D\sum_\alpha D_+(\alpha_{(1)}(D_-
\alpha_{(2)}(E)))~=~\sum_\alpha{}^1\nabla_D(\alpha_{(1)})(\alpha_{(2)}(E)) \\
& = & \sum_\alpha\alpha_{(2)}(E){}^1\nabla_D(\alpha_{(1)})(1)~=~\sum_D\sum_\alpha \alpha_{(2)}(E)D_+(\alpha_{(1)}(D_-)
\end{eqnarray*}
Therefore \eqref{coprod-ours} is also satisfied.
\end{remark}
\begin{lemma} The coproduct is an algebra morphism and a morphism of
$R$-$R$-bimodules.
\end{lemma}
\begin{proof}  The fact that the coproduct is a morphism of $R$-$R$-bimodules is an easy
consequence of the linearity
properties of $\langle-,-\rangle_{1,2}$ (see Lemma \ref{ref-A.6-51}).

We check that $\Delta(1)=1\otimes 1$. This means
\[
\epsilon({}^1\nabla_D{}^2\nabla_E(1))=\langle D,1\rangle_1\langle E,1\rangle_2
=\epsilon(D)\epsilon(E)
\]
(for the last equality we use \eqref{ref-A.4-56}).
We compute
\begin{align*}
\epsilon({}^1\nabla_D{}^2\nabla_E(1))&=\epsilon({}^1\nabla_D(1\!\!1_2
(\epsilon(E))))\qquad \text{(Lemma \ref{ref-A.5-50})}\\
&=\langle 1\!\!1_2
(\epsilon(E)),D \rangle_1\\
&=\epsilon(E) \langle 1,D \rangle_1 \qquad\text{\eqref{ref-A.1-52}}\\
&=\epsilon(E)\epsilon(D)\qquad \eqref{ref-A.4-56}
\end{align*}

We now prove that the coproduct is compatible with multiplication.
We compute
\begin{align*}
\sum_{\alpha\beta}\langle (\alpha\beta)_{(1)},D\rangle_1
\langle E,(\alpha\beta)_{(2)}\rangle_2
&=\epsilon({}^1\nabla_D{}^2\nabla_E(\alpha\beta))\\
&=\sum_{D,E}\epsilon({}^1\nabla_{D_{(1)}}{}^2\nabla_{E_{(1)}}(\alpha))\epsilon({}^1\nabla_{D_{(2)}}{}^2\nabla_{E_{(2)}}(\beta))\\
&=\sum_{D,E,\alpha,\beta} \langle\alpha_{(1)}, D_{(1)}\rangle_1\langle\alpha_{(2)}
,E_{(1)}\rangle_2\langle\beta_{(1)}, D_{(2)}\rangle_1\langle\beta_{(2)}
,E_{(2)}\rangle_2\\
&=\sum_{\alpha,\beta}\langle \alpha_{(1)}\beta_{(1)} ,D\rangle_1
\langle\alpha_{(2)}\beta_{(2)}, E\rangle_2\qed
\end{align*}
\def\qed{}\end{proof}
\begin{lemma} \label{ref-A.12-60} One has the following formul\ae\
\begin{align*}
{}^1\nabla_D\alpha&=1\!\!1_1(\langle \alpha_{(1)},D\rangle_1)\alpha_{(2)}\\
{}^2\nabla_D\alpha&=\alpha_{(1)}1\!\!1_2(\langle\alpha_{(2)} ,D\rangle_2)
\end{align*}
Hence in particular for $D=1$ we get the counit axioms.
\begin{align*}
\alpha&=(1\!\!1_1\circ \epsilon)(\alpha_{(1)})\alpha_{(2)}\\
\alpha&=\alpha_{(1)}(1\!\!1_2\circ \epsilon)(\alpha_{(2)})\\
\end{align*}
\end{lemma}
\begin{proof} To prove for example the first formula we show that
both sides give the same results when applying $\langle-,E\rangle_2$.
We compute
\begin{align*}
\langle 1\!\!1_1(\langle\alpha_{(1)},D\rangle_1)\alpha_{(2)},E\rangle_2
&=\langle \alpha_{(1)},D\rangle_1\langle\alpha_{(2)},E\rangle_2\\
&=\epsilon ({}^1\nabla_D {}^2\nabla_E \alpha)\\
&=\epsilon( {}^2\nabla_E{}^1\nabla_D  \alpha)\\
&=\langle {}^1\nabla_D  \alpha,E\rangle_2
\end{align*}
The second formula is proved in the same way.
\end{proof}
\begin{lemma} The coproduct on ${\mathrm J}_R L$ is coassociative.
\end{lemma}
\begin{proof}
We compute the
two sides of
\[
{}^2\nabla_E{}^1\nabla_D(\alpha)={}^1\nabla_D{}^2\nabla_E(\alpha)
\]
using the formul\ae\ from Lemma \ref{ref-A.12-60}. For the left hand side we find
\begin{align*}
\sum_\alpha {}^2\nabla_E{}^1\nabla_D(\alpha)&=1\!\!1_1(\langle\alpha_{(1)},D\rangle_1)
{}^2\nabla_E(\alpha_{(2)})\\
&=\sum_\alpha 1\!\!1_1(\langle\alpha_{(1)} ,D\rangle_1)
1\!\!1_2(\langle \alpha_{(2)(2)},E\rangle_2)\alpha_{(2)(1)}
\end{align*}
For the right hand side we find
\begin{align*}
\sum_\alpha {}^1\nabla_D{}^2\nabla_E(\alpha)&=1\!\!1_2(\langle\alpha_{(2)},E\rangle_2)
{}^1\nabla_D(\alpha_{(1)})\\
&=\sum_\alpha 1\!\!1_2(\langle\alpha_{(2)},E\rangle_2)
1\!\!1_1(\langle\alpha_{(1)(1)},D\rangle_1)\alpha_{(1)(2)}
\end{align*}
so that we get
\[
\sum_\alpha 1\!\!1_1(\langle D,\alpha_{(1)}\rangle_1)\alpha_{(2)(1)}1\!\!1_2(\langle
E,\alpha_{(2)(2)}\rangle_2)
=\sum_\alpha 1\!\!1_1(\langle D,\alpha_{(1)(1)}\rangle_1)\alpha_{(1)(2)}1\!\!1_2(\langle
E,\alpha_{(2)}\rangle_2)
\]
Since this is true for any $D,E$ we deduce by passing to associated
graded objects and invoking Lemma \ref{ref-A.3-48}
\[
\sum_\alpha\alpha_{(1)}\otimes \alpha_{(2)(1)}\otimes \alpha_{(2)(2)}
=\sum_\alpha \alpha_{(1)(1)}\otimes \alpha_{(1)(2)}\otimes \alpha_{(2)}
\]
which is precisely coassociativity.
\end{proof}
The antipode is defined using a similar formula as for the coproduct
\[
\langle S\alpha, D\rangle_1=\langle \alpha,D\rangle_2
\]
Once again the non-degeneracy of the pairings $\langle-,-\rangle_{1,2}$
implies that we obtain an
invertible map $S:{\mathrm J}_RL\r {\mathrm J}_RL$.
\begin{lemma}
$S$ is an algebra morphism which furthermore exchanges the actions
of $R$ on $\mathrm{J}_R L$ through $1\!\!1_1$ and $1\!\!1_2$.
\end{lemma}
\begin{proof} The fact that $S$ exchanges the two $R$-actions follows
from the linearity property of the  pairings $\langle-,-\rangle_{1,2}$
(see Lemma \ref{ref-A.6-51}).

The fact that $S$ in an algebra morphism follows in a similar
way a for the comultiplication.
\end{proof}
To verify the properties of the antipode we need the following formula.
\begin{lemma} \label{ref-A.15-61} One has for $D\in \mathrm U_RL$, $\alpha\in {\mathrm
J}_RL$
\begin{equation}
\label{ref-A.5-62}
\sum_{D,\alpha}
\langle\alpha_{(1)}, D_{(1)}\rangle_1\langle\alpha_{(2)}, D_{(2)}\rangle_2
=D(\epsilon(\alpha))
\end{equation}
\end{lemma}
\begin{proof}  We first observe that by definition
\[
\langle \alpha_{(1)},D_{(1)}\rangle_1\langle\alpha_{(2)} ,D_{(2)}\rangle_2
=\epsilon({}^1\nabla_{D_{(1)}}{}^2\nabla_{D_{(2)}}(\alpha))
\]
We first claim that that \eqref{ref-A.5-62} is multiplicative in
in $D$.
Assume that \eqref{ref-A.5-62} is correct for $D,E\in \mathrm U_R L$. The we
claim it is also correct for $DE$.
\begin{align*}
\epsilon({}^1\nabla_{(DE)_{(1)}}{}^2\nabla_{(DE)_{(2)}}(\alpha))
&=
\epsilon({}^1\nabla_{(D_{(1)}E_{(1)}}{}^2\nabla_{D_{(2)}E_{(2)}}(\alpha))\\
&=\epsilon({}^1\nabla_{(D_{(1)}}{}^1\nabla_{E_{(1)}}{}^2\nabla_{D_{(2)}}{}^2\nabla_{E_{(2)}}(\alpha))\\
&=\epsilon({}^1\nabla_{(D_{(1)}}{}^2\nabla_{D_{(2)}}{}^1\nabla_{E_{(1)}}{}^2\nabla_{E_{(2)}}(\alpha))\\
&=\epsilon(D(\epsilon({}^1\nabla_{E_{(1)}}{}^2\nabla_{E_{(2)}}(\alpha))))\qquad
\text{(induction)}\\
&=D(\epsilon({}^1\nabla_{E_{(1)}}{}^2\nabla_{E_{(2)}}(\alpha)))\\
&=DE(\epsilon(\alpha))\qquad \text{(induction)}
\end{align*}
Hence it suffices to look at the cases $D=r\in R$ and $D=l\in L$. These
are easy verifications.
\end{proof}
\begin{lemma}
We have
\begin{align}
\label{ref-A.6-63}
\sum_\alpha \alpha_{(1)} S(\alpha_{(2)})&=(1\!\!1_1\circ\epsilon)(\alpha)\\
\sum_\alpha S(\alpha_{(1)}) \alpha_{(2)}&=(1\!\!1_2\circ\epsilon)(\alpha)
\label{ref-A.7-64}
\end{align}
\end{lemma}
\begin{proof} For \eqref{ref-A.6-63} we compute
\begin{align*}
\sum_{\alpha}\langle\alpha_{(1)} S(\alpha_{(2)}), D\rangle_1&=
\sum_{\alpha,D}\langle \alpha_{(1)},D_{(1)}\rangle_1
\langle S(\alpha_{(2)}) ,D_{(2)}\rangle_1\\
&=\sum_{\alpha,D}\langle \alpha_{(1)},D_{(1)}\rangle_1
 \langle \alpha_{(2)} ,D_{(2)}\rangle_2\\
&=D(\epsilon(\alpha))\qquad  \text{(Lemma \ref{ref-A.15-61})}
\end{align*}
and
\begin{align*}
  \langle D,1\!\!1_1\circ \epsilon(\alpha)\rangle_1&
  =\langle D\epsilon(\alpha),1\rangle_1\qquad \text{(Lemma \ref{ref-A.6-51})}\\
  &=\epsilon(D\epsilon(\alpha)) \qquad \text{(Lemma \ref{ref-A.7-54})}\\
  &=D(\epsilon(\alpha))
\end{align*}
The proof for \eqref{ref-A.7-64} is similar (one uses the cocommutativity
of $\mathrm U_R L$).
\end{proof}
Finally we verify:
\begin{lemma} One has $S^2=\Id_{{\mathrm J}_R L}$.
\end{lemma}
\begin{proof} The proof is based on the following computation. On the one hand
\[
\begin{aligned}
\sum_\alpha S^2(\alpha_{(1)})S(\alpha_{(2)})\alpha_{(3)}&=\sum_\alpha
S^2(\alpha_{(1)})(1\!\!1_2\circ\epsilon)(\alpha_{(2)})\\
&=\sum_\alpha S^2(\alpha_{(1)}(1\!\!1_2\circ\epsilon)(\alpha_{(2)}))\\
&=S^2(\alpha);
\end{aligned}
\]
and on the other hand
\[
\begin{aligned}
\sum_\alpha S^2(\alpha_{(1)})S(\alpha_{(2)})\alpha_{(3)}&=\sum_\alpha
S(S(\alpha_{(1)})\alpha_{(2)})\alpha_{(3)}\\
&=\sum_\alpha S((1\!\!1_2\circ\epsilon)(\alpha_{(1)}))\alpha_{(2)}\\
&=\sum_\alpha (1\!\!1_1\circ\epsilon)(\alpha_{(1)})\alpha_{(2)}\\
&=\alpha.
\end{aligned}
\]
We have used the coassociativity, the counit axioms and the fact that $S$ is
an algebra morphism which intertwines the actions $1\!\!1_1$ and $1\!\!1_2$ of $R$ on
${\mathrm J}_RL$.
\end{proof}
%\bibliography{mybibs}
%\bibliographystyle{amsabbrv}
\def\cprime{$'$} \def\cprime{$'$} \def\cprime{$'$}
\ifx\undefined\bysame
\newcommand{\bysame}{\leavevmode\hbox to3em{\hrulefill}\,}
\fi

\end{document}